\documentclass[12pt]{amsart}

\usepackage[utf8]{inputenc}
\usepackage{amsfonts}
\usepackage{amsthm}
\usepackage{amssymb}
\usepackage{amsmath}
\usepackage{amscd}
\usepackage{latexsym,dsfont}
\usepackage{bbm}
\usepackage{mathrsfs}

\usepackage{times}
\usepackage{microtype}
\usepackage{setspace}
\usepackage[margin=1.2in]{geometry}

\usepackage{cite}

\usepackage[colorlinks=true, pdfstartview=FitV, linkcolor=blue,
citecolor=blue, urlcolor=blue]{hyperref}

\usepackage{mathtools}

\newtheorem{theorem}{Theorem}[section]
\newtheorem{lemma}[theorem]{Lemma}
\newtheorem{prop}[theorem]{Proposition}
\newtheorem{cor}[theorem]{Corollary}

\newtheorem{defn}[theorem]{Definition}
\newtheorem{example}[theorem]{Example}

\theoremstyle{definition}

\theoremstyle{remark}
\newtheorem{remark}[theorem]{Remark}

\numberwithin{equation}{section}
 \allowdisplaybreaks

\def\Xint#1{\mathchoice
   {\XXint\displaystyle\textstyle{#1}}%
   {\XXint\textstyle\scriptstyle{#1}}%
   {\XXint\scriptstyle\scriptscriptstyle{#1}}%
   {\XXint\scriptscriptstyle\scriptscriptstyle{#1}}%
   \!\int}
\def\XXint#1#2#3{{\setbox0=\hbox{$#1{#2#3}{\int}$}
     \vcenter{\hbox{$#2#3$}}\kern-.5\wd0}}

\def\avgint{\Xint-}

\makeatletter
\DeclareFontFamily{OMX}{MnSymbolE}{}
\DeclareSymbolFont{MnLargeSymbols}{OMX}{MnSymbolE}{m}{n}
\SetSymbolFont{MnLargeSymbols}{bold}{OMX}{MnSymbolE}{b}{n}
\DeclareFontShape{OMX}{MnSymbolE}{m}{n}{
    <-6>  MnSymbolE5
   <6-7>  MnSymbolE6
   <7-8>  MnSymbolE7
   <8-9>  MnSymbolE8
   <9-10> MnSymbolE9
  <10-12> MnSymbolE10
  <12->   MnSymbolE12
}{}
\DeclareFontShape{OMX}{MnSymbolE}{b}{n}{
    <-6>  MnSymbolE-Bold5
   <6-7>  MnSymbolE-Bold6
   <7-8>  MnSymbolE-Bold7
   <8-9>  MnSymbolE-Bold8
   <9-10> MnSymbolE-Bold9
  <10-12> MnSymbolE-Bold10
  <12->   MnSymbolE-Bold12
}{}

\let\llangle\@undefined
\let\rrangle\@undefined
\DeclareMathDelimiter{\llangle}{\mathopen}%
                     {MnLargeSymbols}{'164}{MnLargeSymbols}{'164}
\DeclareMathDelimiter{\rrangle}{\mathclose}%
                     {MnLargeSymbols}{'171}{MnLargeSymbols}{'171}
\makeatother

\DeclareMathOperator{\supp}{supp}

\DeclareMathOperator*{\esssup}{ess\,sup}

\DeclareMathOperator{\grad}{\nabla}

\newcommand{\op}{{\mathrm{op}}}

\newcommand{\R}{\mathbb{R}}
\newcommand{\rn}{{\mathbb{R}^n}}

\newcommand{\Ss}{\mathcal S}

\newcommand{\bB}{\mathbf B}
\newcommand{\bBo}{\overline{\bB}}

\newcommand{\cB}{\mathcal B}

\newcommand{\vece}{\mathbf e}
\newcommand{\vecf}{\mathbf f}
\newcommand{\vecg}{\mathbf g}

\newcommand{\vecv}{\mathbf v}
\newcommand{\vecx}{\mathbf x}

\newcommand{\vecu}{\mathbf u}

\newcommand{\A}{\mathcal{A}}
\newcommand{\K}{\mathcal{K}}
\newcommand{\D}{\mathcal{D}}

\newcommand{\Md}{\mathcal{M}_d}
\newcommand{\Sd}{\mathcal{S}_d}
\newcommand{\W}{\mathcal W}

\newcommand{\avgL}{\textit{\L}} 

\title{On off-diagonal operators in matrix-weighted spaces}

\author{D. Cruz-Uribe OFS, A. Laukkarinen, and K. Moen}

\address{David Cruz-Uribe, OFS \\
Dept. of Mathematics \\
University of Alabama \\
 Tuscaloosa, AL 35487, USA}

\email{dcruzuribe@ua.edu}

\address{Aapo  Laukkarinen \\
Aalto University \\
Department of Mathematics and Systems Analysis \\
P.O. Box 11100 \\
FI-00076 Aalto, Finland}

\email{aapo.laukkarinen@aalto.fi}

\address{Kabe Moen \\
Dept. of Mathematics \\
University of Alabama \\
 Tuscaloosa, AL 35487, USA}

 \email{kabe.moen@ua.edu}

 \thanks{The first author is partially supported by a Simons Foundation
  Travel Support for Mathematicians Grant and by NSF Grant DMS-2349550. The second author was supported by the Research Council of Finland through grant 364208. The authors would like to thank Tuomas Hyt\"onen for his support which allowed the second author to visit the University of Alabama, where this project was begun.}


\begin{document}

\begin{abstract}
    In this paper we prove matrix-weighted inequalities for fractional operators and their commutators.  We do so by developing the theory of convex body domination for such operators.  Using this approach we prove quantitative estimates for the fractional integral operator (or Riesz potential) and its commutators, and prove matrix-weighted Gagliardo-Nirenberg-Sobolev inequalities for vector-valued functions.
\end{abstract}
\maketitle

\section{Introduction}
\label{section:introduction}
 The goal of this paper is to prove off-diagonal,  matrix-weighted norm inequalities for fractional operators.  To put our results into context, we provide a brief history.  Many operators in harmonic analysis are off-diagonal in the sense that they map $L^p\to L^q$, $p\neq q$. One of the most fundamental examples is the fractional integral (or Riesz potential) $I_\alpha$, $0<\alpha<n$, defined by
\[
    I_\alpha f(x)\coloneqq\int_{\R^n}\frac{f(y)}{|x-y|^{n-\alpha}}\, dy.
\]
The boundedness of this operator dates back to the work of Hardy and Littlewood \cite{HL1928someproperties} and Sobolev \cite{Sob38}. Weighted estimates estimates for $I_\alpha$ were proved by Muckenhoupt and Wheeden \cite{muckenhoupt-wheeden74}. They showed that 
\begin{equation*}
    \|wI_\alpha f\|_{L^q(\R^n)}\leq C\|wf\|_{L^p(\R^n)}
\end{equation*}
if and only if $\frac1p-\frac1q=\frac{\alpha}{n}$ and
\[
    [w]_{A_{p,q}}\coloneqq
    \sup_Q \bigg(\avgint_Qw^q\,dx\bigg)^{\frac1q} \left(\avgint_Q w^{-p'}\,dx\right)^\frac{1}{p'}<\infty.
\]

Motivated by the regularity results of Astala, Ivaniec and Saksman \cite{MR1815249}, the question of the optimal dependence on the weight in weighted norm inequalities  has received considerable attention.  For the fractional integral this question becomes the problem of finding the dependence on the $A_{p,q}$ characteristic $[w]_{A_{p,q}}$. Lacey, P\'erez, Torres and the third author showed that the constant is of the form 
$C[w]_{A_{p,q}}^{(1-\frac{\alpha}{n})\max\{p',q\}}$ and this is sharp in the sense that no smaller power can be taken on the $A_{p,q}$ characteristic.

The technique referred to as  sparse domination, developed by Lerner \cite{Lern2012,lerner-IMRN2012} for Calder\'on-Zygmund operators, is often the cleanest way of proving these sharp weighted estimates. In the case of $I_\alpha$, this technique amounts to showing that $I_\alpha f$ is dominated pointwise by a finite collection of sparse dyadic fractional operators, 
\begin{equation}\label{Eq:SparseDominationforIalpha}
   I^\mathcal{S}_\alpha f(x)\coloneqq \sum_{Q\in\mathcal{S}}|Q|^\frac{\alpha}{n}\avgint_Qf(y)\,dy\,\mathbbm 1_Q(x),
\end{equation}
where $\mathcal{S}$ is a sparse collection of cubes in a dyadic grid $\mathcal{D}$. The first explicit proof of \eqref{Eq:SparseDominationforIalpha} was given by the first and third author in \cite{MR3224572}.  The proof for fractional integrals is much simpler than for singular integrals.  By a dyadic decomposition one can show that $I_\alpha$ is pointwise comparable to a finite linear combination of dyadic operators 
\[
    I_\alpha^\mathcal{D} f(x)\coloneqq \sum_{Q\in\mathcal{D}}|Q|^\frac{\alpha}{n}\avgint_Qf(y)\,dy\,\mathbbm 1_Q(x).
\]
Then, by summing Calder\'on-Zygmund cubes on many levels, one can show that the dyadic operator is bounded by a sparse operator, i.e., 
\[
    I_\alpha^\mathcal{D} f(x)\leq C I_\alpha^\mathcal{S} f(x).
\]

\medskip

Matrix-weighted inequalities are a natural extension of  scalar-weighted estimates, where the operator now acts on vector-valued functions componentwise. Such inequalities were first considered for the Hilbert transform by Nazarov, Treil, and Volberg~\cite{MR1428988,MR1428818,MR1478786,MR1423034} and then for general Calder\'on-Zygmund operators by Christ and Goldberg~\cite{MR1813604,MR2015733}.  Matrix-weighted inequalities for fractional integrals were first considered by Isralowitz and the third author~\cite{MR4030471}, who showed that 
\begin{equation}\label{Eq:IMmatrixweightedineq}
    \|WI_\alpha\vecf\|_{L^q(\R^n,\R^d)}\leq C [W]_{\mathcal{A}_{p,q}}^{(1-\frac{\alpha}{n})p'+q-1}\|W\vecf\|_{L^p(\R^n,\R^d)},
\end{equation}
where
\[
[W]_{\A_{p,q}} \coloneqq \sup_{Q} \bigg( \avgint_Q \left( \avgint_{Q} |W(x) W^{-1}(y)|^{p'}_{\op} \,dy \right)^{\frac{q}{p'}} \,dx \bigg)^{\frac{1}{q}}.
\]
Their proof essentially involved carrying out the sparse domination argument from \cite{MR3224572} discussed above with the matrix weight inserted into the terms, making it essentially a scalar argument.

Another approach to proving matrix-weighted inequalities involves a generalization of the sparse domination technique.  The idea is to replace the averages with convex sets that precisely capture the information on the size and direction of the vector-valued functions. This idea was introduced by Petermichl, Nazarov, Treil and Volberg~\cite{MR3689742}; they refer to this technique as convex body domination.  They showed that if $T$ is a Calder\'on-Zygmund operator, given a vector-valued function $\vecf\in L^\infty_c(\R^n,\R^d)$, there exists a sparse family (not necessarily dyadic) such that
\[ T\vecf(x) 
\in C \sum_{Q\in \Ss} \llangle \vecf\rrangle_{\avgL^1(Q)}\, \mathbbm 1_Q(x).\]
They used convex body domination to show the matrix-weighted inequality 
\begin{equation}\label{Eq:MatrixAtwoboundforCZO}
    \|WT\vecf\|_{L^2(\R^n,\R^d)}\leq C[W]_{\mathcal{A}_2}^3\|W\vecf\|_{L^2(\R^n,\R^d)}.
\end{equation}
This bound was recently  shown to be sharp when $T$ is the Hilbert transform by Domolevo, Petermichl, Treil, and Volberg \cite{DPTV-2024}. It is important to compare this to the sharp $\mathcal{A}_2$ bound in the scalar case for Calder\'on-Zygmund operators \cite{MR2912709}, where the exponent $3$ in \eqref{Eq:MatrixAtwoboundforCZO} is replaced by $2$.  Convex body domination was also used to prove quantitative bounds on $L^p$, $1<p<\infty$, by Isralowitz and the first and third authors~\cite{MR3803292}. Their result matches \eqref{Eq:IMmatrixweightedineq} when $\alpha=0$ and gives the sharp exponent $3$ when $p=2$.  It is not known if their result is the best possible for all $p$.

\medskip

\subsection*{Main results}
Our first main result  is an extension of convex body domination to off-diagonal operators. As a special case,  we show the following.

\begin{theorem} \label{thm:I-alpha-sparse-intro}
Given $0<\alpha<n$, there exists $0<\eta<1$ such that, for every $\vecf \in L^\infty_c(\rn,\R^d)$, there exists $\eta$-sparse families $\Ss_i$, $i=1,\dots,3^n$, in different dyadic lattices such that for almost every $x\in \rn$,
\begin{equation}\label{Eq:CBDforfracint}
    I_\alpha \vecf(x) 
\in C \sum_{i=1}^{3^n}\sum_{Q\in \Ss_i} |Q|^{\frac{\alpha}{n}}\llangle \vecf\rrangle_{\avgL^1(Q)}\, \mathbbm 1_Q(x).
\end{equation}
\end{theorem}

In light of the proof strategy of \eqref{Eq:IMmatrixweightedineq} in \cite{MR4030471} it would be reasonable to guess that a generalization of the argument in \cite{MR3224572} would also work for the proof of \eqref{Eq:CBDforfracint}. An idea inspired by \cite{mb-2026} is to try to repeat the argument for convex-set valued functions, with inequalities replaced by inclusions. This approach initially works, and shows that $I_\alpha \vecf$ is contained in a dyadic operator analogous to the right-hand side of~\eqref{Eq:CBDforfracint}.  However, we were unable to complete this approach by directly estimating the dyadic operator with the sparse operator.  It is an open problem to show that the argument from the scalar case can be extended in this way.

Instead, we develop an approach inspired by the work of Lerner \cite{Ler2016OnPointwise}, where weak type inequalities are used to achieve sparse domination. We then generalize a bootstrapping argument due to Hyt\"onen to prove convex body domination for off-diagonal operators.  This approach is quite flexible and allows us to prove convex body domination for a number of classes of operators. In particular, our domination result applies to fractional singular integrals, and to fractional operators whose kernels satisfy H\"ormander-type conditions.  This operator class has been considered by a number of authors and it includes, for example, rough fractional integrals and fractional Schr\"odinger operators.

\medskip

We next consider three applications of off-diagonal convex body domination.    First, we use it to prove quantitative weighted norm inequalities for off-diagonal operators.  For instance, we use Theorem~\ref{thm:I-alpha-sparse-intro} to prove an improved matrix-weighted inequality for the fractional integral.

\begin{theorem} \label{thm:I-alpha-bound-intro}
Given $0<\alpha<n$ and $1<p<\frac{n}{\alpha}$, define $q$ by $\frac{1}{p}-\frac{1}{q}=\frac{\alpha}{n}$.  If $W\in \A_{p,q}$, then
for all $\vecf\in L^p(\rn,W)$,
\begin{equation} \label{eqn:mfi1-intro}
 \|WI_\alpha \vecf\|_{L^q(\R^n,\R^d)} 
\leq C[W]_{\A_{p,q}}^{\min\{(1-\frac{\alpha}{n})p'+q, p'+(1-\frac{\alpha}{n})q\}-1}\|W\vecf\|_{L^p(\R^n,\R^d)}. 
\end{equation}
\end{theorem}


The weight dependence in Theorem~\ref{thm:I-alpha-bound-intro} coincides with~\eqref{Eq:IMmatrixweightedineq} when $q\leq p'$ and is better when $q>p'$. As we noted above, the sharp exponent in the scalar case is $(1-\frac{\alpha}{n})\max\{p',q\}$, so our exponent is worse for all values of $p$ and $q$. It is not clear what the best exponent in the matrix case should be. This question is similar to the problem of the best exponent in the matrix case for Calder\'on-Zygmund operators discussed above. If we formally set $p=q=2$ and $\alpha=0$, then the exponent in \eqref{eqn:mfi1-intro} coincides with the sharp $\A_2$ bound \eqref{Eq:MatrixAtwoboundforCZO}.

\medskip

We prove Theorem \ref{thm:I-alpha-bound-intro} as a byproduct of our second application,  In the scalar case, given a  function $b$ the classical commutator $[b,T]$ of an operator $T$ is defined by 
\[[b,T]f\coloneqq bTf-T(bf).\] 
For vector-valued functions, we replace the scalar function by a matrix $B$:
\[[B,T]\vecf\coloneqq BT\vecf-T(B\vecf).\] 
To recapture the scalar operator,  let $B=bI_d$, where $I_d$ is the $d\times d$ identity matrix.
Matrix-weighted inequalities for commutators have been considered by a number of authors~\cite{MR4269407, MR4454483, MR3687948, CI2022twomatrixweightedfracint}.  

More recently, there has been interest in proving estimates for higher order commutators, such as the iterated commutator
\[
    C_B^k(T)\coloneqq[B,C_B^{k-1}(T)],\qquad C^0_B(T)\coloneqq T.
\]
Our approach to this problem is based on the observation by several authors that commutators can be written as  vector-valued functions (see~\cite{MR4269407,MR4454483, Lau2026rough}). This structure lets convex body domination  be applied to commutators. We build upon the recent work of Hyt\"onen \cite{hyt2024remarksonCBD}, who defined in the scalar case generalized commutators
\begin{equation*}\label{Eq:IntroGenCommutator}
    \sum_{i=1}^ka_iT(b_if),
\end{equation*}
where $a_i$ and $b_i$, $i=1,\ldots,k$, are scalar functions.  He used convex body domination to prove
 diagonal $L^p\to L^p$ estimates. The power of this generalization is that particular choices of the functions  $a_i,\,b_i$ yield the classical commutator, the iterated commutator, and multi-symbol iterated commutators.  

Starting from an observation in~\cite[Remark~7.12]{hyt2024remarksonCBD}, we extend these ideas to vector-valued functions $\vecf$ and matrix functions $A_i,\,B_i$, and consider the generalized commutators
\[ \sum_{i=1}^k A_iT(B_i\vecf). \]
We prove off-diagonal inequalities for these operators, but we note that our results also hold in the diagonal case $p=q$.  Following the approach used by Isralowitz and the first and third authors~\cite{MR3803292},  we first prove two-matrix-weight inequalities using the so-called bump conditions developed in the scalar case by a number of authors.    We give sufficient conditions for the  $L^p(W)\to L^q(V)$ boundedness of the generalized commutator.   See Theorem~\ref{thm:orliczbumps} and its corollaries for a precise statement of our result and applications to specific commutators. Here we note that for the first order commutator, our results coincide with earlier results by Cardenas and Isralowitz~\cite{CI2022twomatrixweightedfracint} and  Isralowitz, Pott, and Treil~\cite{MR4454483}.

We then prove one matrix-weight inequalities from the two-weight case.  To do so we use the so-called power bumps and the sharp reverse H\"older inequality for scalar $\A_p$ weights.  For instance, as a consequence we obtain the following  result for iterated commutators.

\begin{theorem}\label{thm:MatWeightCommutatorBoundIntro}
    Let $\alpha,p,q$ be as in Theorem \ref{thm:I-alpha-bound-intro}. Let $k$ be a nonnegative integer and suppose that $W\in \A_{p,q}$ and $b\in \operatorname{BMO}$. Then we have 
    \[
        \|WC^k_b(I_\alpha)\vecf\|_{L^q(\R^n,\R^d)}\leq C\|b\|_{BMO}^k [W]_{\A_{p,q}}^{\beta}\|W\vecf\|_{L^p(\R^n,\R^d)},
    \]
    where 
    \[
        \beta\coloneqq k\max\left\{p',q\right\}+\min\{(1-\frac{\alpha}{n})p'+q, p'+(1-\frac{\alpha}{n})q\}-1.
    \]
\end{theorem}

When $k=0$, the commutator reduces to the fractional integral operator $I_\alpha$, which gives us the proof of Theorem~\ref{thm:I-alpha-bound-intro}.  The contribution of the commutators to the exponent $\beta$, $k\max\left\{p',q\right\}$, is the same as in the scalar case: see~\cite[Corollary 5.2]{BMMST2019CommutatorBoundsUnified}.  This suggests that if the constant for $I_\alpha$ in Theorem~\ref{thm:I-alpha-bound-intro} is sharp, then so is the constant for the commutator.

\begin{remark}
    As we were working on this paper, we learned that Isralowitz, Rivera-R\'\i os, and S\'aez-Rivas~\cite{IRS2026convexbodydominationcommutator} had developed a different approach to matrix-weighted norm inequalities for commutators of diagonal operators (i.e., when $\alpha=0$).  There is some overlap between their results and ours:  see the discussion in Remarks~\ref{remark:irrsr0} and~\ref{remark:IRRSR} below.
\end{remark}

As the final application of convex body domination, we prove matrix-weighted Gagliardo-Nirenberg-Sobolev inequalities.  For instance, we prove the following result.

\begin{theorem} \label{thm:GNS-intro}
Fix $1<p<n$ and define the Sobolev exponent $p^*$ by $\frac{1}{p}-\frac{1}{p^*}= \frac{1}{n}$.  If $W \in \A_{p,p^*}$, then for all $\vecf$ such that $D\vecf \in L^p(W)$, 
\begin{equation}\label{Eq:IntroGNSineq}
    \|W\vecf\|_{L^{p^*}(\R^n,\R^d)}\leq C[W]_{\A_{p,p^*}}^{\min\{\frac{p'}{n'}+p^*,\frac{p^*}{n'}+p'\}-1} \|WD\vecf\|_{L^p(\R^n,\R^d)},
\end{equation}
where $D\vecf$ is the derivative matrix of $\vecf$.  
\end{theorem}

 Theorem~\ref{thm:GNS-intro} was first proved in \cite{MR4030471} with the larger exponent $\frac{p'}{n'}+p^*-1$. Their proof mirrored the fairly long proof of \eqref{Eq:IMmatrixweightedineq}. Here, however, we are able to use  convex body domination directly to prove \eqref{Eq:IntroGNSineq}. Moreover, using our approach we are able to prove a more general version of this result, replacing $D\vecf$ by the fractional derivative matrix  $D^\alpha\vecf$, $0 <\alpha<1$.  

\medskip

\medskip

\subsection*{Organization}
The remainder of this paper is organized as follows. In Section \ref{section:prelim}, we recall some definitions and known results that will be used later in the paper. 

In Section \ref{section:bootstrapping}, we prove convex body domination for general off-diagonal operators.  Our main result is Theorem~\ref{thm:pointwiseCBDfromweaktypebounds}, which is a generalization of the convex body domination for Calder\'on-Zygmund operators proved in~\cite{MR3689742}.  To prove it, we first develop a general theory of bootstrapping pointwise scalar inequalities for off-diagonal operators to convex body inclusion relations for these operators applied to vector-valued functions:  see Proposition~\ref{prop:pointwisebootstrap}.  This result can be thought of as a pointwise version of the form bound proved in~\cite{hyt2024remarksonCBD}.  Theorem~\ref{thm:pointwiseCBDfromweaktypebounds} is stated and proved for off-diagonal operators $T$ that satisfy weak-type $(r,s)$ inequalities for $1\leq r,\,s<\infty$, and such that the truncated maximal operator introduced by Lerner~\cite{Ler2016OnPointwise} satisfies similar weak-type bounds.  We then prove Theorem~\ref{thm:I-alpha-bound-intro} as a consequence of Theorem~\ref{thm:pointwiseCBDfromweaktypebounds}.  In Corollary~\ref{Cor:MTalpha-bd} we prove convex body sparse bounds for fractional singular integrals; these operators arise naturally in studying Gagliardo-Nirenberg-Sobolev inequalities and will be applied in Section~\ref{section:sobolev}.   Finally, in Corollary~\ref{Cor:MTalpha-hormander-bd} we prove convex body domination for a family of fractional operators that satisfy a H\"ormander-type kernel condition.  These operators have been considered by a number of different authors, and we give some examples and references at the end of the section.

In Section \ref{section:commutators} we prove quantitative, matrix-weighted inequalities that follow from convex body domination. We will first consider commutators;  the bounds for the operators themselves are then a consequence of these commutator bounds. We first  prove a sufficient two-weight bump condition for the generalized commutator in Theorem~\ref{thm:orliczbumps}.  We then describe the particular choices of the $A_i,\,B_i$ that yield the first order commutator, the iterated commutator, and multi-symbol commutators. Next we use  Theorem~\ref{thm:orliczbumps} to prove one-weight inequalities. With the appropriate choice of bump functions, we use the sharp reverse Hölder inequality to prove a generalized version of Theorem \ref{thm:MatWeightCommutatorBoundIntro}:  see Theorem~\ref{thm:MatWeightCommutatorBound}.

In Section \ref{section:sobolev} we prove matrix-weighted Gagliardo-Nirenberg-Sobolev inequalities for vector-valued functions.  We first prove Theorem~\ref{thm:GNS-intro} as, essentially, an application of Theorem~\ref{thm:I-alpha-bound-intro}, and then prove a more general version for fractional derivatives in the same way:  see Theorem~\ref{Thm:MatrixWeightedFractionalGradientBound}.

\section{Preliminaries}
\label{section:prelim}

Throughout this paper, we will use the following notation.
 In Euclidean space, the constant $n$ will denote
the dimension of $\R^n$, which will be the domain of our functions.  The value
$d$ will denote the dimension of vector and set-valued functions.     For $1\leq p \leq \infty$, $L^p(\R^n)$ will denote the Lebesgue space
 of scalar functions, and $L^p(\R^n,\R^d)$ will denote the Lebesgue
 space of vector-valued functions.  Given a cube $Q$, we define the normalized $L^p$ norm on $Q$ by
 \[ \|f\|_{\avgL^p(Q)} = \bigg(\avgint_Q |f|^p\,dx\bigg)^{\frac1p}
 = \bigg(\frac{1}{|Q|}\int_Q |f|^p\,dx\bigg)^{\frac1p}. \]

 Given $\vecv=(v_1,\ldots,v_d)^t \in \R^d$, the Euclidean norm of $\vecv$ will
 be denoted by $|\vecv|$ and for $j=1,\ldots,d$, we define $[\vecv ]_j= v_j$.  The standard orthonormal basis in $\R^d$ will
 be denoted by $\{\vece_i\}_{i=1}^d$.  The open unit ball in $\R^d$ 
 will be denoted by ${\mathbf B}$ and its
 closure by $\overline{\mathbf{B}}$.  Matrices will be $d\times d$
 matrices with real-valued entries unless otherwise specified.  The
 set of all such matrices will be denoted by $\Md$.  The set of all
 $d\times d$, symmetric (i.e., self-adjoint), positive semidefinite matrices will be denoted by
 $\Sd$.  We will denote the transpose of a matrix $W$ by $W^*$.  
 
 Constants will be denoted by $C$ and $c$, and their values may change from line to line.  To show the dependence of constants on particular parameters, we will write, for instance, $C(n,d)$.  Some constants will only be given implicitly.  Given
 two quantities $A$ and $B$, we will write $A \lesssim B$, or
 $B\gtrsim A$ if there is a constant $c>0$ such that $A\leq cB$.  If
 $A\lesssim B$ and $B\lesssim A$, we will write $A\approx B$.

\subsection*{Dyadic cubes and sparse sets}
In this section we give the basic properties of dyadic cubes and sparse families.  We follow the approach developed in~\cite{MR4007575}.   By a cube $Q$ we will always mean a cube whose sides are parallel to the coordinate axes.
Given a cube $Q$, let $\mathcal{D}_k(Q)$ be the collection of cubes obtained by subdividing $Q$ into $2^{kn}$ congruent subcubes with disjoint interiors. Let $\mathcal{D}(Q)\coloneqq\bigcup_{k=0}^\infty\mathcal{D}_k(Q)$.

\begin{defn}
We say that a collection of cubes $\mathcal D$ is a dyadic lattice, if it satisfies the following three properties:
\begin{enumerate}
    \item If $Q\in\mathcal D$ and $Q'\in \mathcal D(Q)$, then $Q'\in \mathcal D$.
    \item If $Q',Q''\in\mathcal D$, then there exists $Q\in\mathcal{D}$ such that $Q',Q''\in\mathcal{D}(Q)$.
    \item Every compact set in $\R^d$ is contained in some cube from $\mathcal D$.
\end{enumerate}
We refer to the elements of a dyadic lattice as dyadic cubes.
\end{defn}
We note that the so-called standard dyadic lattice does not satisfy conditions (2) and (3). However, it is fairly straightforward to create  collections of cubes that do satisfy them. This formulation of dyadic lattices is due to \cite[Section 2]{MR4007575}, and we refer the reader there for more details. The following theorem is known as the three lattice trick and the proof can be found in \cite[Theorem 3.1]{MR4007575}.

\begin{theorem} \label{thm:three-lattice}
    For every dyadic lattice $\mathcal{D}$ there exists $3^n$ dyadic lattices $\mathcal{D}_j$, $j=1,\dots,3^n$, such that
    \[
        \{3Q\,\colon\,Q\in \mathcal{D}\}=\bigcup_{j=1}^{3^n}\mathcal{D}_j.
    \]
\end{theorem}

An important class of families of cubes are sparse collections.

\begin{defn}
Given $0<\eta<1$, we say that a collection of cubes $\mathcal{S}\subset\mathcal{D}$ is $\eta$-sparse if for every cube $Q\in \mathcal{Q}$, there exists a subset $E_Q$ such that $|E_Q|\geq \eta|Q|$ and the subsets $E_Q$ are pairwise disjoint. 
\end{defn}
We can relax this definition in the sense that the sparse cubes can come from finitely many different dyadic lattices. This means that we may inflate the cubes in the sparse collection by a factor of $3$. Indeed, by the three lattice trick we can write
\[
    \{3Q\,\colon\,Q\in \mathcal{S}\}=\bigcup_{j=1}^{3^n}\mathcal{S}_j,
\]
where $\mathcal{S}_j$ is a subset of a dyadic lattice and for each cube  $3Q\in \mathcal{S}_j$,
\[
  |E_{3Q}| :=  |E_Q|\geq \eta|Q|=3^{-n}\eta|3Q|.
\]
Thus, the collections $\mathcal{S}_j$ are $3^{-n}\eta$-sparse.

\subsection*{Matrix weights}
We gather together some basic properties of matrix weights.  For further information, see~\cite{mb-2026}.
Given a matrix $W\in \Md$, its operator norm is defined by 
\[ |W|_{\op} =  \sup_{\substack{\vecv\in \R^d\\ |\vecv|=1}} |W\vecv|.\]
The following lemmas are very useful for estimating operator norms.

\begin{lemma}{\cite[Lemma 3.2]{MR1928089}}\label{opNorm:equiv}
If $\{\vecv_1, \ldots, \vecv_d\}$ is any orthonormal basis in $\R^d$, then for any  matrix $W\in \Md$ and $r>0$, 
\[
	 \bigg(\sum_{i=1}^d |W \vecv_i|^r\bigg)^{\frac{1}{r}} \approx |W|_{\op}, 
\]
where the implicit constants depend only on $d$ and $r$.
\end{lemma}

The following result is well-known.

\begin{lemma}\label{SelfAdjointCommutes}
Given $U,\,V \in \Ss_d$, $|UV|_{\op} = |VU|_{\op}$.
\end{lemma}


The following lemma is referred to as the Cordes inequality~\cite[Lemma~5.1, p.~24]{MR890743}.

\begin{lemma} \label{lemma:cordes}
Given matrices $U,\,V \in \Sd$ and $0\leq s \leq 1$, $|U^sV^s|_{\op}\leq |UV|_{\op}^s$. Consequently, if $r\geq 1$,
$|UV|_{\op}^r \leq |U^rV^r|_{\op}$.
\end{lemma}

A matrix weight is a measurable function $W\colon\R^n\to\mathcal S_d$ that is invertible almost everywhere, and whose entries are finite almost everywhere. The matrix-weighted $L^p$ space $L^p(W)$ consists of measurable functions $\vecf\colon\R^n\to\R^d$ such that
\[
    \|W\vecf\|_{L^p(\R^n,\R^d)}<\infty.
\]

\begin{remark}
    We note that our definition of the matrix-weighted space $L^p(W)$ is different from the original definitions in, for instance~\cite{hytonen-perez-treil-volbergP, MR1428988, MR1428818, MR2015733, MR1813604}, where the norm was written as $\|W^{1/p}\vecf\|_{L^p(\R^n,\R^d)}$.  We have instead chosen to follow the approach taken in~\cite{mb-2026}.  One reason for this is that it preserves the left-openness of the matrix $\A_p$ classes, which is not the case for the original definition.  (See~\cite{Bow, dcu-mp-2025} for details.)
\end{remark}

Given $p,\,q$,  $1<p\le q<\infty$, we say a matrix weight $W$ is in the class $\A_{p,q}$, denoted by $W\in \A_{p,q}$, if 
\[
[W]_{\A_{p,q}} = \sup_{Q} \left( \avgint_Q \left( \avgint_{Q} |W(x) W^{-1}(y)|^{p'}_{\op} \,dy \right)^{\frac{q}{p'}} \,dx \right)^{\frac{1}{q}}  < \infty.
\]
We refer to the quantity $[W]_{\A_{p,q}}$ as the $\A_{p,q}$ characteristic of $W$.
We can also characterize matrix  $\A_{p,q}$ weights using the so-called reducing operators.  For a proof of the next result, 
see~\cite[Section~2]{MR4030471}. 

\begin{prop} \label{prop:frac-reducing}
   Given $0<\alpha<n$ and $1\leq p<\frac{n}{\alpha}$,
define $q$ by $\frac{1}{p}-\frac{1}{q}=\frac{\alpha}{n}$.   Let $W\in
\A_{p,q}$.  Then for
every cube $Q$ there exist constant matrices $\W_Q^q$ and
$\overline{\W}_Q^{p'}$ such that for every $\vecv\in \R^d$,
\[ |\W_Q^q v|\approx \bigg(\avgint_Q |W(x)\vecv|^q\,dx\bigg)^{\frac{1}{q}},
  \quad
  |\overline{\W}_Q^{p'}| \approx
  \bigg(\avgint_Q |W^{-1}(x)\vecv|^{p'}\,dx\bigg)^{\frac{1}{p'}}. \]
Moreover,
\[ [W]_{\A_{p,q}}^R = \sup_Q |\W_Q^q \overline{\W}_Q^{p'}|_{\op}
  \approx [W]_{\mathcal{A}_{p,q}}. \]
In each case, the implicit constants depend only $d$, $p$, and
$\alpha$.  
\end{prop}

\begin{remark} \label{remark:Apq-dual}
    It follows from Proposition~\ref{prop:frac-reducing} and Lemma~\ref{SelfAdjointCommutes} that $[W]_{\A_{p,q}}^R=[W^{-1}]_{\A_{q',p'}}^R$, and so $W\in \A_{p,q}$ if and only if $W^{-1} \in \A_{q',p'}$.
\end{remark}

Recall that a weight $w$ satisfies the reverse H\"older inequality with exponent $s>1$ and a constant $C>1$ if
for every cube $Q$,
\[  \avgint_Q w^s\,dx \leq C\bigg(\avgint_Q w\,dx\bigg)^s.  
\]
Any weight $w$ in the scalar $A_p$ satisfies a  reverse H\"older inequality.  In fact, from~\cite[Theorem~2.3]{HPR2012SharpRH} and \cite[Prop.~2.2]{MR3092729} that we have that this inequality holds with $C=2$ and $1<s\leq \frac{1}{c(n)[w]_{A_p}-1}$.  Further, we have that if $w$ is in scalar $\A_{p,q}$, then by H\"older's inequality, $w^q \in A_q$ and $[w]_{\A_{p,q}}^q\geq[w^q]_{A_q}$.  We will use this fact in conjunction with the following result due to Isralowitz and the third author~\cite[Cor.~3.3]{MR4030471}.

\begin{lemma} \label{lemma:matrix-RHI}
If $W \in\A_{p,q}$, then for any vector $\vecv \in \R^d$, $|W\vecv|$ is in scalar $\A_{p,q}$ and $[|W\vecv|]_{\A_{p,q}} \leq c[W]_{\A_{p,q}}$.  Consequently, there exists a constant $\tau(n)$ such that $|W\vecv|^q$ satisfies the sharp reverse H\"older inequality
\[  \avgint_Q |W\vecv|^{sq}\,dx \leq 2\bigg(\avgint_Q |W\vecv|^q\,dx\bigg)^s, \quad 1<s\leq 1+\frac{1}{\tau(n)[W]_{\A_{p,q}}^q-1}.
\]
\end{lemma}

\begin{remark}
    We note that in comparing our results to those in~\cite{MR4030471}, the reader must take into account the different definitions of scalar and matrix $A_{p,q}$ used there.
\end{remark}

\subsection*{Convex bodies}
In this section we give some  basic definitions and results about convex sets that appear in convex body domination.  For a more detailed discussion, see~\cite{mb-2026}.   Recall that a set $A \in \R^d$ is symmetric if $x\in A$ if and only if $-x\in A$. Let $\K(\R^d)$ denote the set of all closed, bounded, convex, and symmetric sets in $\R^d$.  Given a set $A\in \K(\R^d)$ and a vector $\vecv \in \R^d$, we define the inner product
\[ A \cdot \vecv := \sup\{ \vecu\cdot \vecv : \vecu \in A \}.  \]
Since $A$ is symmetric, we always have that $A \cdot \vecv \geq 0$.  
Given any two sets $A,\,B \in \K(\R^d)$, define their Minkowski sum to be the set
\[ A+ B = \{ u+v : u \in A, v\in B \}.\]
Note that $A+B$ is again convex.  Given a collection of convex sets $\{A_i\}_{i=1}^\infty$, we can define the infinite Minkowski sum
\[  \sum_{i=1}^\infty A_i.  \]
It can be shown with suitable assumptions on the sets $A_i$ that this sum converges in Hausdorff metric to a convex set.  (See also~\cite[Lemma~2.5]{MR3689742}.)

Given a vector valued function $\vecf\in L^r$, we define the set-valued $L^r$ average, $1\leq r<\infty$, by 
\[ \llangle \vecf \rrangle_{\avgL^r(Q)} 
=\bigg\{ \avgint_Q k\vecf\,dx : \|k\|_{L^{r'}(Q,\R)}\leq 1 \bigg\}. \]
Denote the collection of functions $k$ that satisfy the selection condition $\|k\|_{L^{r'}(Q,\R)}\leq 1$ by $\bBo(r',Q)$.  Since this set is closed, bounded, symmetric and convex,  it is straightforward to prove that if $\vecf\in L^r(Q,\R^d)$, then $\llangle \vecf \rrangle_{\avgL^r(Q)}\in \mathcal{K}(\R^d)$.  See~\cite{MR4245601}. 

The following is the classical John ellipsoid theorem.
\begin{theorem} \label{thm:john-ellipsoid}
    Let $K\in \K(\R^d)$. Then there exists a closed ellipsoid $\mathcal{E}$ centered at the origin such that $\mathcal E\subset K\subset \sqrt{d}\mathcal E$.
\end{theorem}
We call the ellipsoid in Theorem~\ref{thm:john-ellipsoid} the John ellipsoid. By the definition of an ellipsoid we can write $\mathcal{E}=A\overline{\mathbf{B}}$, where $A\in \Md$. It follows from the polar decomposition of $A$  that we can replace $A$ with a matrix $\tilde A \in \Ss_d$.  Moreover, we have that
\begin{equation} \label{eqn:john-ellipsoid}
    \mathcal{E}=\left\{\vecx\in\R^d\,\colon\,\sum_{i=1}^d\left(\frac{|\mathbf x\cdot\vecv_i|}{\sigma_i}\right)^2\leq1\right\},
\end{equation}
where $\{\vecv_i\}_{i=1}^d$ are the eigenvectors of $\tilde A$ and $\{\sigma_i\}_{i=1}^d$ are the corresponding eigenvalues. If for some $i$, $\sigma_i=0$, we interpret this by setting
\[
    \frac{|\vecx\cdot \vecv_i|}{0}\coloneqq \begin{cases}
        0,  &\text{if } |\vecx\cdot \vecv_i|=0,\\
        \infty, &\text{if } |\vecx\cdot \vecv_i|>0.
    \end{cases}
\]
We refer to the orthonormal vectors $\{\vecv_i\}_{i=1}^d$ as the principal axes of the ellipsoid $\mathcal{E}$.

\subsection*{Young functions}
We recall some basic definitions and results from the theory of Orlicz spaces.  For further details and proofs, see~\cite{Rao-Ren}. Let $\Phi\colon[0,\infty)\to[0,\infty)$ be a Young function, i.e., an increasing convex function such that $\Phi(0)=0$ and $\Phi(t)/t\to\infty$ as $t\to\infty$.
The localized Orlicz norm is
\[
    \|f\|_{\avgL^\Phi(Q)}\coloneqq \inf\left\{\lambda>0\,\colon\,\avgint_Q\Phi\left(\frac{|f|}{\lambda}\right)\leq1\right\}.
\]
This norm has the following rescaling property:  given $r\geq 1$, define $\Phi_r(t)=\Phi(t^r)$.  Then 
\begin{equation} \label{eqn:rescaling}
\||f|^r\|_{\avgL^\Phi(Q)}^{\frac{1}{r}} = \|f\|_{\avgL^{\Phi_r}(Q)}. 
\end{equation}

The associate function of a Young function $\Phi$ is defined by 
\[
    \bar\Phi(t)\coloneqq \sup\{st-\Phi(s)\,\colon\,s>0\},
\]
and the pair $(\Phi,\bar\Phi)$ satisfies the generalized H\"older inequality
\begin{equation}\label{eq:orliczholder}
    \avgint_Q fg\leq 2\|f\|_{\avgL^\Phi(Q)}\|g\|_{\avgL^{\bar\Phi(Q)}}.
\end{equation}
Given $1<p,q<\infty$ say $\Phi$ satisfies the $\mathcal{B}_{p,q}$ condition, denoted by $\Phi\in \mathcal B_{p,q}$, if 
\begin{equation} \label{eqn:Bpq-cond}
    B_{p,q}(\Phi)\coloneqq\left(\int_1^\infty\frac{\Phi(t)^\frac{q}{p}}{t^q}\frac{dt}{t}\right)^\frac{1}{q}<\infty.
\end{equation}
For brevity, we set  $\mathcal B_p\coloneqq \mathcal B_{p,p}$. By~\cite[Theorem~3.3]{DCU-KMfrac2}, if $\Phi\in \mathcal B_{p,q}$, then the fractional maximal Orlicz operator
\[
    M_{\alpha,\Phi}f(x)\coloneqq \sup_{Q\ni x}|Q|^\frac{\alpha}{n}\|f\|_{\avgL^\Phi(Q)}
\]
satisfies
\begin{equation}\label{Eq:FracOrliczMaximalEst}
    \|M_{\alpha,\Phi}f\|_{L^q(\R^n)}\lesssim B_{p,q}(\Phi)\|f\|_{L^p(\R^n)},
\end{equation}
where the implicit constants depend on $n,\,p,\,q,\,\alpha,\, \Phi$.

\section{Off-diagonal convex body domination}
\label{section:bootstrapping}

In this section we prove our main convex body domination theorem, Theorem~\ref{thm:pointwiseCBDfromweaktypebounds}, and prove Theorem~\ref{thm:I-alpha-sparse-intro} as a consequence.  We also show how it can be used to bound more general off-diagonal operators.

\subsection*{The bootstrapping argument}
We first show that scalar-valued sparse bounds for off-diagonal operators can be lifted to prove convex body domination for the same operators applied to vector-valued functions. Our results generalize the form bounds in~\cite{hyt2024remarksonCBD} and our proofs are adapted from the ones found there.  While they are similar, we give the details to show the changes required.

We start with a lemma that gives a useful condition for identifying elements of $\llangle \vecf\rrangle_{\avgL^{r}(Q)}$.
\begin{lemma}\label{lem:convexbodytest}
    Let $Q$ be a cube, and let $\{\vecv_i\}_{i=1}^d$ be the principal axes of the John ellipsoid $\mathcal{E}$ of $\llangle \vecf\rrangle_{\avgL^r(Q)}$. If $|\mathbf x\cdot\vecv_i|\leq \frac{1}{d}\|\vecf\cdot\vecv_i\|_{\avgL^r(Q)}$ for every $i=1,\dots,d,$ then $\vecx\in\llangle \vecf\rrangle_{\avgL^r(Q)}$.
\end{lemma}
\begin{proof}
    It suffices to show that $\vecx \in\mathcal{E}$, which by \eqref{eqn:john-ellipsoid} is equivalent to showing that
    \[
        \sum_{i=1}^d\left(\frac{|\mathbf x\cdot\vecv_i|}{\sigma_i}\right)^2\leq1,
    \]
    where $\sigma_i$ are the eigenvalues of the matrix that defines $\mathcal{E}$.
    By our hypothesis on $\vecx$ and duality, we have 
    \begin{multline*}
        \sum_{i=1}^d\left(\frac{|\mathbf x\cdot\vecv_i|}{\sigma_i}\right)^2
        \leq\frac{1}{d^2} \sum_{i=1}^d\left(\frac{\|\vecf\cdot\vecv_i\|_{\avgL^r(Q)}}{\sigma_i}\right)^2
        =\frac{1}{d^2}\sum_{i=1}^d\sup_{\phi\in \bBo(r',Q)}\left(\frac{|\avgint_{Q}(\vecf\cdot\vecv_i)\,\phi\,dx|}{\sigma_i}\right)^2\\
        \leq\sup_{\phi\in\bBo(r',Q)}\frac{1}{d}\sum_{i=1}^d\left(\frac{|\avgint_{Q}(\vecf\cdot\vecv_i)\,\phi\,dx|}{\sigma_i}\right)^2
        =\sup_{\phi\in\bBo(r',Q)}\frac{1}{d}\sum_{i=1}^d\left(\frac{|(\avgint_{Q}\vecf\,\phi \,dx)\cdot\vecv_i|}{\sigma_i}\right)^2.
    \end{multline*}
   By the definition we have that for all such $\phi$,  $\avgint_{Q}\vecf\,\phi\,dx \in \llangle \vecf\rrangle_{\avgL^r(Q)}\subset \sqrt d \mathcal{E}$; thus, $\frac{1}{\sqrt d}\avgint_{Q}\vecf\,\phi\,dx\in \mathcal{E}$, which means that
    \[
        \frac{1}{d}\sum_{i=1}^d\left(\frac{|(\int_{\R^n}\vecf\,\phi)\cdot\vecv_i|}{\sigma_i}\right)^2\leq1.
    \]
    If we combine these estimates we get the desired inclusion.
\end{proof}

\begin{remark}
    Lemma~\ref{lem:convexbodytest} is also true in a general Banach-space setting developed in~\cite{hyt2024remarksonCBD}. Let $X(Q)$ be a Banach space. The convex bodies are then defined by
    \[
        \llangle \vecf\rrangle_{X(Q)}\coloneqq\{\langle \vecf,\phi\rangle_{(X(Q),X^*(Q))}\,\colon \phi\in \overline{\mathbf{B}}(X^*(Q))\},
    \]
    where $\langle \vecf,\phi\rangle_{(X(Q),X^*(Q))}\coloneqq(\langle f_1,\phi\rangle_{(X(Q),X^*(Q))},\dots,\langle f_d,\phi\rangle_{(X(Q),X^*(Q))})\in\R^d$. In this setting, the proof of Lemma \ref{lem:convexbodytest} works verbatim, if we replace $\avgL^r(Q)$ with $X(Q)$ and $\avgint_Q\vecf \phi\,dx$ with $\langle \vecf,\phi\rangle_{(X(Q),X^*(Q))}$.  In~\cite{hyt2024remarksonCBD} form domination is proved in this general setting.  It is an open problem to show that the pointwise domination we prove can be done in this generality.
\end{remark}

The following proposition is a  generalization of \cite[Lemma 3.2]{MR3689742}.
\begin{prop}\label{prop:pointwisebootstrap}
    Let $Q_0$ be a dyadic cube and let $T_Q$ be linear operators associated with a cube $Q$. Suppose that for all $f\in \avgL^r(3Q_0)$ there exists a disjoint collection $\mathcal{G}\subset\mathcal{D}(Q_0)$ with the following properties:
    \begin{enumerate}
        \item $\sum_{Q\in\mathcal{G}}|Q|\leq \varepsilon|Q_0|$;
        \item if $\bar{\mathcal{G}}\subset\mathcal{D}(Q_0)$ is a disjoint collection such that $\bigcup_{\bar Q\in\bar{\mathcal{G}}}\bar Q\supset \bigcup_{Q\in\mathcal{G}}Q$ and any $Q\in \mathcal{G}$ is covered by some $\bar Q\in
        \bar{\mathcal{G}}$, then for almost every $x$, 
        \[
            |T_{Q_0}f(x)-\sum_{\bar Q\in\bar{\mathcal{G}}}T_{\bar Q}f(x)|\leq C_0 \|f\|_{\avgL^r(3Q_0)}.
        \]
    \end{enumerate}
    Then for all $\vecf\in\avgL^r(Q,\R^d)$ there exists a disjoint collection $\hat{\mathcal{G}}$ such that $\sum_{\hat Q\in\hat{\mathcal{G}}} |\hat Q|\leq \varepsilon d|Q_0|$, and for almost every $x$, 
        \[
            T_{Q_0}\vecf(x)-\sum_{\hat Q\in\hat{\mathcal{G}}}T_{\hat Q}\vecf(x)\in C_0d \llangle \vecf\rrangle_{\avgL^r(3Q_0)}.
        \]
\end{prop}
\begin{proof}
    By assumption, for each $i=1,\dots,d$ there exists a disjoint collection $\mathcal{G}_i\subset \mathcal{D}(Q_0)$ such that $\sum_{Q_i\in\mathcal{G}_i}|Q_i|\leq\varepsilon|Q_0|$ and 
    \begin{equation}\label{eq:bootscalarineq}
        |T_{Q_0}(\vecf\cdot\vecv_i)-\sum_{Q_i\in\mathcal{G}_i}T_{Q_i}(\vecf\cdot\vecv_i)|\leq C_0 \|\vecf\cdot\vecv_i\|_{\avgL^r(3Q_0)},
    \end{equation}
    where the $\{\vecv_i\}_{i=1}^d$ are as in Lemma \ref{lem:convexbodytest}.
    Let $\hat{\mathcal{G}}$  be the maximal disjoint  cubes contained in $\bigcup_{i=1}^d\mathcal{G}_i$. Then 
    \[
        \sum_{\hat Q\in\hat{\mathcal{G}}}|\hat Q|\leq\sum_{i=1}^d\sum_{Q_i\in\mathcal{G}_i}| Q_i|\leq\varepsilon d|Q_0|
    \]
    and \eqref{eq:bootscalarineq} holds for each $i$ with $\hat{\mathcal{G}}$ in place of $\mathcal{G}_i$. Thus, for each $i$ we have
    \begin{equation*}
        \Big|\Big(T_{Q_0}\vecf-\sum_{\hat Q\in\hat{\mathcal{G}}}T_{\hat Q}\vecf\,\Big)\cdot\vecv_i\Big|
        =|T_{Q_0}(\vecf\cdot\vecv_i)-\sum_{\hat Q\in\hat{\mathcal{G}}}T_{\hat Q}(\vecf\cdot\vecv_i)|
        \leq C_0 \|\vecf\cdot\vecv_i\|_{\avgL^r(3Q_0)}.
    \end{equation*}
    By Lemma \ref{lem:convexbodytest} we have that 
    \[
            \frac{1}{C_0d}\Big(T_{Q_0}\vecf-\sum_{\hat Q\in\hat{\mathcal{G}}}T_{\hat Q}\vecf\,\Big)\in \llangle \vecf\rrangle_{\avgL^r(3Q_0)};
    \]
    this completes the proof.
\end{proof}

\subsection*{Estimates from weak type bounds}
We now use Proposition~\ref{prop:pointwisebootstrap} to prove convex body domination.    
We will  prove a general theorem building on the approach of Lerner. Given an operator $T$, define the maximal operator
\[
    M_Tf(x)\coloneqq \sup_{Q\ni x}\esssup_{y\in Q}|T(\mathbbm 1_{\R^n\setminus3Q}f)(y)|.
\]
We refer to $M_T$ as the Lerner maximal operator associated to $T$. Our goal is to prove convex body domination for $T$ with the assumption that $T$ and $M_T$ are bounded maps, $L^r\to L^{s,\infty}$. We first show that the hypotheses of Proposition~\ref{prop:pointwisebootstrap} hold for these operators.

\begin{prop}\label{prop:scalarbounds}
    Fix $1 \leq r,\, s< \infty$.  Suppose $T$ is a linear operator such that $T$ and $M_T$ are bounded maps, $L^r\to L^{s,\infty}$. Given any dyadic cube $Q_0$, $f\in \avgL^r(3Q_0)$, and $0<\varepsilon<1$, there exists $\mathcal{G}\subset\mathcal{D}(Q_0)$ such that
    \[
        \sum_{Q\in\mathcal{G}}|Q|\leq \varepsilon|Q_0|,
    \]
    and if $\bar{\mathcal{G}}\subset\mathcal{D}(Q_0)$ is a disjoint collection such that $\bigcup_{\bar Q\in\bar{\mathcal{G}}}\bar Q\supset \bigcup_{Q\in\mathcal{G}}Q$ and any $Q\in \mathcal{G}$ is covered by some $\bar Q\in\bar{\mathcal{G}}$, then there exists a constant $C\coloneqq C(n,s,r,T)$ such that
    \[
        |T(\mathbbm 1_{3Q_0}f)\mathbbm 1_{Q_0}-\sum_{\bar Q\in\bar{\mathcal{G}}} T(\mathbbm 1_{3\bar Q}f)\mathbbm 1_{\bar Q}|\leq C \varepsilon^{-\frac{1}{s}}|Q_0|^\frac{s-r}{sr}\|f\|_{\avgL^r(3Q_0)}\mathbbm 1_{Q_0}
    \]
    almost everywhere.
\end{prop}
\begin{proof}
    For $\lambda>0$ define the set
    \[
        E_\lambda\coloneqq Q_0\cap\{|T(\mathbbm 1_{3Q_0}f)|>\lambda \text{ or } M_T(\mathbbm 1_{3Q_0}f)>\lambda\};
    \]
    the exact value of $\lambda$ will be fixed below.  Let $\mathcal{G}\subset\mathcal{D}(Q_0)$ be the collection of maximal cubes $Q$ that satisfy 
    \[
        \frac{|Q\cap E_\lambda|}{|Q|}>2^{-n-1}.
    \]
    It follows from the weak-type bounds that
    \begin{align*}
        |E_\lambda|&\leq |\{|T(\mathbbm 1_{3Q_0}f)|>\lambda\}|+|\{|M_T(\mathbbm 1_{3Q_0}f)|>\lambda\}|\\
        &\leq \frac{1}{\lambda^s}3^\frac{ns}{r}(\|T\|^s_{L^r\to L^{s,\infty}}+\|M_T\|^s_{L^r\to L^{s,\infty}})\|f\|^s_{\avgL^r(3Q_0)}|Q_0|^\frac{s}{r}\\
        &\leq 2\Big(\frac{1}{\lambda}3^\frac{n}{r}c(r,s,T)\|f\|_{\avgL^r(3Q_0)}|Q_0|^\frac{1}{r}\Big)^s,
    \end{align*}
    where $c(r,s,T)\coloneqq\max\{\|T\|_{L^r\to L^{s,\infty}},\|M_T\|_{L^r\to L^{s,\infty}}\}$. Therefore, 
    \begin{multline*}
        |G|\coloneqq \bigg|\bigcup_{Q\in\mathcal{G}}Q\bigg|
        =|\{M^\mathcal{D}(\mathbbm 1_{E_\lambda})>2^{-n-1}\}|\\
        \leq 2^{n+1}|E_\lambda|
        \leq 2^{n+2}\Big(\frac{1}{\lambda}3^\frac{n}{r}c(r,s,T)\|f\|_{\avgL^r(3Q_0)}|Q_0|^\frac{1}{r}\Big)^s,
    \end{multline*}
    where $M^\mathcal{D}$ is the Hardy-Littlewood maximal operator defined with respect to the dyadic grid $\D$. If we let
    \[
        \lambda\coloneqq 2^\frac{n+2}{s}3^\frac{n}{r}c(r,s,T)\varepsilon^{-\frac{1}{s}}|Q_0|^\frac{s-r}{sr}\|f\|_{\avgL^r(3Q_0)},
    \] 
    then we get $\sum_{Q\in\mathcal{G}}|Q|\leq \varepsilon|Q_0|$. Note that this implies that the cubes $Q\in\mathcal{G}$ are proper subcubes of $Q_0$.  

\smallskip

    Now let  $\bar{\mathcal{G}}\subset\mathcal{D}(Q_0)$ be a disjoint collection such that $\bar G\coloneqq\bigcup_{\bar Q\in\bar{\mathcal{G}}}\bar Q\supset \bigcup_{Q\in\mathcal{G}}Q=G$ and any $Q\in \mathcal{G}$ is covered by some $\bar Q\in\bar{\mathcal{G}}$. We estimate as follows:
    \begin{equation}\label{eq:simpledecomp}
        |T(\mathbbm 1_{3Q_0}f)\mathbbm 1_{Q_0}-\sum_{\bar Q\in\bar{\mathcal{G}}} T(\mathbbm 1_{3\bar Q}f)\mathbbm 1_{\bar Q}|
         \leq 
         |T(\mathbbm 1_{3Q_0}f)|\mathbbm 1_{Q_0\setminus\bar G}
         +\sum_{{\bar Q\in\bar{\mathcal{G}}}}|T(\mathbbm 1_{3Q_0\setminus 3\bar Q}f)|\mathbbm 1_{\bar Q}.
    \end{equation}
    Since $M^\mathcal{D}(\mathbbm 1_{E_\lambda})\geq\mathbbm 1_{E_\lambda}$ almost everywhere, we see that $E_\lambda$ is contained in $G$, except perhaps for a subset of measure zero. Hence, up to a set of measure zero, we have $\bar G\supset E_\lambda$, which implies
    \begin{equation}\label{eq:firstterm}
        |T(\mathbbm 1_{3Q_0}f)|\mathbbm 1_{Q_0\setminus\bar G}\leq \lambda\mathbbm 1_{Q_0\setminus\bar G}
    \end{equation}
    almost everywhere.

    To estimate the second term on the right-hand side of \eqref{eq:simpledecomp}, note that since the cubes in $\mathcal{G}$ are proper subcubes of $Q_0$, we have by maximality that every parent cube $P_Q$ of $Q\in\mathcal{G}$ satisfies
    \[
        \frac{|P_Q\cap E_\lambda|}{|P_Q|}\leq2^{-n-1}.
    \]
    This in turn implies that
    \[
        \frac{|Q\cap E_\lambda|}{|Q|}\leq\frac{|P_Q\cap E_\lambda|}{2^{-n}|P_Q|}\leq \frac{1}{2}.
    \]
    Thus, for any $Q\in\mathcal{G}$,  $|Q\setminus E_\lambda|\geq \frac{1}{2}|Q|>0$. Therefore, any $Q\in\mathcal{G}$ intersects $Q_0\setminus E_\lambda$. In particular, this property holds for any $\bar Q\in\bar{\mathcal{G}}$. Indeed, any $\bar Q\in\bar{\mathcal{G}}$ either covers some $Q\in\mathcal{G}$ or is  disjoint from $G$.
    If $\bar Q$ covers some $Q\in\mathcal{G}$, then clearly $\bar Q$ intersects $Q_0\setminus E_\lambda$. If $\bar Q$ and $G$ are disjoint, then again by maximality
    \[
        \frac{|\bar Q\cap E_\lambda|}{|\bar Q|}\leq2^{-n-1}.
    \]
    Therefore, any $\bar Q\in \bar{\mathcal{G}}$ intersects $Q_0\setminus E_\lambda$.
    Hence, there exists a point $z\in \bar Q$ such that
    \[
        \esssup_{y\in \bar Q}|T(\mathbbm 1_{3Q_0\setminus3\bar Q})(y)|\leq  
        M_T(\mathbbm 1_{3Q_0}f)(z)\leq \lambda;
    \]
    consequently,
    \[
        |T(\mathbbm 1_{3Q_0\setminus3\bar Q})|\mathbbm 1_{\bar Q}\leq \lambda\mathbbm 1_{\bar Q}
    \]
    almost everywhere. If we combine this with \eqref{eq:simpledecomp} and \eqref{eq:firstterm}, we get
    \[
        |T(\mathbbm 1_{3Q_0}f)\mathbbm 1_{Q_0}-\sum_{\bar Q\in\bar{\mathcal{G}}} T(\mathbbm 1_{3\bar Q}f)\mathbbm 1_{\bar Q}|
         \leq  \lambda\mathbbm 1_{Q_0}= C\varepsilon^{-\frac{1}{s}}|Q_0|^\frac{s-r}{sr}\|f\|_{\avgL^r(3Q_0)}\mathbbm 1_{Q_0}
    \]
    as desired.
\end{proof}

We are now ready to prove convex body domination for $T$.

\begin{theorem}\label{thm:pointwiseCBDfromweaktypebounds}
    Let $r,s,T$ be as in Proposition \ref{prop:scalarbounds}. Then there exists $0<\eta<1$ such that, for every $\vecf \in L^\infty_c(\rn,\R^d)$, there exist  $\eta$-sparse collections $\mathcal{S}_i$, $i=1,\dots,3^n$, in different dyadic lattices, and a constant $C\coloneqq C(n,s,r,T)$  such that for almost every $x\in\R^n$,
    \[
        T\vecf(x)\in \sum_{i=1}^{3^n}\sum_{Q\in\mathcal{S}_i}|Q|^\frac{s-r}{sr}\llangle \vecf\rrangle_{\avgL^r(Q)}\mathbbm 1_{Q}(x).
    \]
\end{theorem}
\begin{proof}
    We partition $\R^n$ with maximal dyadic cubes $\{Q_0^k\}_{k=1}^\infty$ such that $Q_0^k\not\supset\supp \vecf$. If $\vecf$ is non-trivial, since it has compact support, we have that  the maximal cubes exist since any compact set is contained in some (large) dyadic cube. Consequently, we have that $\supp \vecf\subset 3Q^k_0$ for each $k$, so we may write
    \[
        T\vecf=\sum_{k=1}^\infty (T\vecf)\mathbbm 1_{Q_0^k}= \sum_{k=1}^\infty T(\mathbbm 1_{3Q_0^k}\vecf)\mathbbm 1_{Q_0^k}.
    \]
    If we combine Proposition \ref{prop:pointwisebootstrap} and Proposition \ref{prop:scalarbounds} with $\frac{\varepsilon}{d}$ in place of $\varepsilon$, we get that there exists  a constant $C:=C(d,n,s,r,\varepsilon,T)$ and a disjoint collection $\mathcal{G}_k\subset \mathcal{D}(Q_0^k)$ such that $\sum_{Q\in\mathcal{G}_k}|Q|\leq \varepsilon|Q_0^k|$ and 
\[
    T(\mathbbm 1_{3Q_0^k}\vecf)\mathbbm 1_{Q_0^k}\in C|Q_0^k|^\frac{s-r}{sr}\llangle \vecf\rrangle_{\avgL^r(3Q_0^k)}\mathbbm 1_{Q_0^k}+\sum_{Q\in\mathcal{G}_k}T(\mathbbm 1_{3Q}\vecf)\mathbbm 1_{Q}.
\]
We now repeat this  argument on each $Q\in \mathcal{G}_k$ in place of $Q_0^k$ and iterate. After the $N$-th iteration, the size of the cubes is at most $\varepsilon^N|Q_0^k|$, which implies that as $N\to\infty$ the terms produced by the iteration will tend to $0$ almost everywhere. Therefore, in the limit the iteration argument gives us that
\begin{equation*}
    T(\mathbbm 1_{3Q_0^k}\vecf)\mathbbm 1_{Q_0^k}\in C\sum_{Q\in\tilde{\mathcal{S}_k}}|Q|^\frac{s-r}{sr}\llangle \vecf\rrangle_{\avgL^r(3Q)}\mathbbm 1_{Q},
\end{equation*}
where the collection $\tilde{\mathcal{S}_k}\subset\mathcal{D}(Q_0^k)$ is $(1-\varepsilon)$-sparse. Then  $\tilde{\mathcal{S}}=\bigcup_k\tilde{\mathcal{S}_k}$ is also $(1-\varepsilon)$-sparse, and we have
\[
    T\vecf\in C\sum_{Q\in\tilde{\mathcal{S}}}|Q|^\frac{s-r}{sr}\llangle \vecf\rrangle_{\avgL^r(3Q)}\mathbbm 1_{Q}\subset C\sum_{Q\in\tilde{\mathcal{S}}}|3Q|^\frac{s-r}{sr}\llangle \vecf\rrangle_{\avgL^r(3Q)}\mathbbm 1_{3Q}.
\]
Finally, by Theorem~\ref{thm:three-lattice} (the three lattice trick)  we have that $\{3Q\,\colon\,Q\in\tilde{\mathcal S}\}=\bigcup_{i=1}^{3^n}\mathcal{S}_i$, where the $\mathcal S_i$ are $3^{-n}(1-\varepsilon)$-sparse collections in different dyadic lattices. Therefore, we have 
\[
    T\vecf\in C\sum_{i=1}^{3^n}\sum_{Q\in\mathcal{S}_i}|Q|^\frac{s-r}{sr}\llangle \vecf\rrangle_{\avgL^r(Q)}\mathbbm 1_{Q}.
\]
\end{proof}

\subsection*{Sparse bounds for fractional operators}
In this section we first prove convex body domination for fractional integrals, Theorem~\ref{thm:I-alpha-sparse-intro}, as an immediate consequence of Theorem~\ref{thm:pointwiseCBDfromweaktypebounds}.  We then prove bounds for a family of more general fractional singular integrals, and for fractional integrals that satisfy H\"ormander conditions.  (We give precise definitions of both classes of operators below.)

We first consider $I_\alpha$, $0<\alpha<n$.  It is well-known (see, for example, Grafakos~\cite[Theorem~6.1.3]{grafakos08b}) that the fractional maximal operator satisfies
\begin{equation} \label{eqn:Ialpha-bd} 
I_\alpha : L^r(\R^n) \rightarrow L^{s,\infty}(\R^n), 
\end{equation}
where $r=1$ and $s=\frac{n}{n-\alpha}$.  Thus, Theorem~\ref{thm:I-alpha-sparse-intro} follows at once if we can prove that the
Lerner maximal operator, 
\[ M_{I_\alpha}f(x) = \sup_{Q\ni x} \sup_{y\in Q} |I_\alpha(f\mathbbm 1_{\R^n\setminus{3Q}})(y)|, \]
also satisfies
\begin{equation} \label{eqn:MIalpha-bd} 
M_{I_\alpha} : L^1(\R^n) \rightarrow L^{s,\infty}(\R^n).
\end{equation}

To prove this, fix $x$, $Q$ containing $x$ and $y\in Q$.  If $z\in \R^n\setminus 3Q$, then $|x-z|\approx |y-z|$, where the implicit constants depend only on $n$.  Hence,
    \[ |I_\alpha(f\mathbbm 1_{\R^n\setminus{3Q}})(y)| 
    \leq \int_{\R^n\setminus 3Q} \frac{|f(z)|}{|y-z|^{n-\alpha}}\,dz 
     \leq c\int_{\R^n\setminus 3Q} \frac{|f(z)|}{|x-z|^{n-\alpha}}\,dz 
     = cI_\alpha(|f|)(x). 
\]
Since this estimate is independent of $Q$ and $y$, we have that $M_{I_\alpha}f(x)\lesssim I_\alpha(|f|)(x)$.  It now follows at once from \eqref{eqn:Ialpha-bd} that \eqref{eqn:MIalpha-bd} holds.

\medskip

    Inequalities analagous to~\eqref{eqn:Ialpha-bd} and~\eqref{eqn:MIalpha-bd} also hold for more general fractional singular integrals. Given $0 \leq \alpha < n$, we say that an operator $T_\alpha$, defined on measurable functions, is a fractional singular integral if $T_\alpha$ is bounded from $L^p(\R^n)$ to $L^q(\R^n)$ for some fixed $1 < p \leq q < \infty$, and for any $f \in L^p(\R^n)$ with compact support,
\begin{gather}\label{Eq:genFracInt}
    T_\alpha f(x) = \int_{\R^n} K_\alpha(x,y) f(y) \, dy.
\end{gather}
Here, the kernel $K_\alpha(x,y)$ is  defined for all $(x,y) \in \R^n \times \R^n$, $x \neq y$, and satisfies the  size estimate
\begin{gather} \label{kernel singularity bound}
    |K_\alpha(x,y)| \leq \frac{C_0}{|x-y|^{n-\alpha}}.
\end{gather}

Given this definition, it is immediate that $|T_\alpha f(x)| \leq I_\alpha(|f|)(x)$, and both weak-type inequalities follow at once. Thus, we have convex body domination for such operators.   We record this as a corollary.

\begin{cor} \label{Cor:MTalpha-bd}
Fix $0<\alpha<n$ and let $T_\alpha$ be a fractional singular integral operator. Then there exists $0<\eta<1$ such that, for every $\vecf \in L^\infty_c(\rn,\R^d)$, there exist $\eta$-sparse collections $\mathcal{S}_i$, $i=1,\dots,3^n$, in different dyadic lattices with the property that for almost every $x\in \rn$,
\begin{equation}\label{Eq:CBDforfracOper}
    T_\alpha \vecf(x) 
\in C \sum_{i=1}^{3^n}\sum_{Q\in \Ss_i} |Q|^{\frac{\alpha}{n}}\llangle \vecf\rrangle_{\avgL^1(Q)}\, \mathbbm 1_Q(x).
\end{equation}
\end{cor}

\begin{remark}
    Our definition of a fractional singular integral operator is different from the definition given in~\cite{MR3470665}, where they assume that~\eqref{Eq:genFracInt}  only holds if $x\not\in \supp(f)$ and that the kernel satisfies the smoothness estimate
\[    |K_\alpha(x+h,y) - K_\alpha(x,y)| + |K_\alpha(x,y+h) - K_\alpha(x,y)| \leq C_0 \frac{|h|^\delta}{|x-y|^{n-\alpha+\delta}}
\]
for all $|h| < \frac{1}{2}|x-y|$ and some fixed $\delta > 0$.  In fact, in Section~\ref{section:sobolev} we will apply Corollary~\ref{Cor:MTalpha-bd} to operators which also satisfy this definition. We chose our definition since the pointwise bound on $T_\alpha$ by $I_\alpha$ immediately gives us Corollary~\ref{Cor:MTalpha-bd}.  We leave it as an open problem to show that convex body domination holds for the fractional singular integrals as defined in~\cite{MR3470665}.
\end{remark}

\medskip

Finally, we consider a family of fractional integrals that satisfy H\"ormander-type conditions and which satisfy the hypotheses of Theorem~\ref{thm:pointwiseCBDfromweaktypebounds} with $r\geq 1$.   Fix $0<\alpha<n$ and $1\leq r<\frac{n}{\alpha}$;  we consider a "rough" fractional singular integral $T_\alpha$ defined  by
\[ T_\alpha f(x) = \int_{\R^n} K_\alpha(x-y)f(y)\,dy, \]
where the convolution kernel $K_\alpha$ satisfies the size condition
\begin{equation}\label{Eq:Hormanderkernel1}
    \|\mathbbm 1_{t<|\cdot|\leq 2t}K_\alpha\|_{\avgL^{r'}(B(0,2t))} 
    = \left(\frac{1}{|B(0,2t)|}\int_{t<|x|\leq2t}|K_\alpha(x)|^{r'}\,dx\right)^\frac{1}{r'}\leq Ct^{\alpha-n}, \quad t>0,
\end{equation}
and an $L^{r'}$-H\"ormander condition:  there exist constants $c_r,\,C_r >0$, such that for all $x\in\R^n$ and $R>c_r|x|$,
\begin{equation}\label{Eq:Hormanderkernel2}
   \sum_{m=1}^\infty (2^mR)^{n-\alpha}\|\mathbbm 1_{2^mR<|\cdot|\leq 2^{m+1}R}(K_\alpha(\cdot-x)-K_\alpha)\|_{\avgL^{r'}(B(0,2^{m+1}R))}\leq C_r.
\end{equation}
In \cite[Proposition~1.10, Lemma~3.2]{FRV2022sharp},  they showed that $T_\alpha$ and $M_{T_\alpha}$ are bounded, $L^{r}\to L^{s,\infty}$, where $1\leq r<\frac{n}{\alpha}$ and $s=\frac{rn}{n-\alpha r}$. We note that $\frac{s-r}{sr}=\frac{\alpha}{n}$. Therefore, by Theorem~\ref{thm:pointwiseCBDfromweaktypebounds} we immediately get the following result.

\begin{cor} \label{Cor:MTalpha-hormander-bd}
Fix $0<\alpha<n$ and $1\leq r<\frac{n}{\alpha}$.  Let $T_\alpha$ be defined as above. Then there exists $0<\eta<1$ such that, for every $\vecf \in L^\infty_c(\rn,\R^d)$, there exist $\eta$-sparse collections $\mathcal{S}_i$, $i=1,\dots,3^n$, in different dyadic lattices  with the property that for almost every $x\in \rn$,
\begin{equation}\label{Eq:CBDforfracOper}
    T_\alpha \vecf(x) 
\in C \sum_{i=1}^{3^n}\sum_{Q\in \Ss} |Q|^{\frac{\alpha}{n}}\llangle \vecf\rrangle_{\avgL^r(Q)}\, \mathbbm 1_Q(x).
\end{equation}
\end{cor}

Corollary~\ref{Cor:MTalpha-hormander-bd} applies to a number of different operators.  
If $K_\alpha(x)= |x|^{\alpha-n}$, that is, if $T_\alpha$ is the fractional integral, then these conditions hold for $r=1$.  When $r>1$, define rough fractional integrals with kernels of the form 
\[ K_\alpha(x) = \Omega(x')|x|^{\alpha-n}, \]
where $x'= x/|x|$ and $\Omega\in L^{r'}(S^{n-1})$, $r'>\frac{n}{n-\alpha}$.  In~\cite[Section~2]{FRV2022sharp} they show that if $\Omega$ satisfies a Dini-type continuity condition, then the conditions of Corollary~\ref{Cor:MTalpha-hormander-bd} hold.  
These operators have been considered previously by Bernardis, {\em et al.}~\cite{MR2884902}, Segovia and Torrea~\cite{MR1074151}, Chanillo, {\em et al.}~\cite{MR1247201}, and Ding and Lu~\cite{MR1618714}.  

In~\cite[Remark~1.4]{FRV2022sharp}, the authors note that their results can be extended to fractional singular integrals with non-convolution kernels.  Such operators include fractional operators related to the Riesz transform associated with the Schr\"odinger operator $L=-\Delta+V$, where the potential $V$ satisfies the reverse H\"older condition $RH_{\frac{n}{2}}$.  These are a special case of a class of operators considered by Kurtz~\cite{MR1054087}.  Convex body domination also holds for these operators; we leave the details to the interested reader.

\section{Bounds for generalized commutators}
\label{section:commutators}

In this section we use the convex body domination  in Theorem~\ref{thm:pointwiseCBDfromweaktypebounds} to prove matrix-weighted inequalities for commutators.   To unify our presentation, we will consider the general class of operators $T_\alpha$, $0\leq \alpha <n$,  that satisfy the following convex body domination:
for all $\vecf\in L^\infty_c(\R^d,\R^n)$ there exist sparse collections $\mathcal{S}_i$, $i=1,\dots,3^n$, of dyadic cubes such that
\begin{equation}\label{Eq:CBDcommutatorsection}
    T_\alpha \vecf(x) 
    \in C \sum_{i=1}^{3^n}\sum_{Q\in \Ss_i} |Q|^{\frac{\alpha}{n}}\llangle \vecf\rrangle_{\avgL^r(Q)}\, \mathbbm 1_Q(x).
\end{equation}  
We are primarily interested in the off-diagonal case, but we note that our results hold when $\alpha=0$.  If $T_\alpha$ is an operator that satisfies the hypotheses of Theorem~\ref{thm:pointwiseCBDfromweaktypebounds} for some $1\leq r,\,s<\infty$, then it satisfies~\eqref{Eq:CBDcommutatorsection} with $\frac{\alpha}{n}=\frac{s-r}{sr}$.  In particular, for the fractional integral operator $I_\alpha$ we have $r=1$, $s=\frac{n}{n-\alpha}$. 

We now define the generalized commutator with matrix symbols.  
Let $\mathbf A=(A_1,\dots,A_k)$ and $\mathbf B=(B_1,\dots,B_k)$ be $k$-tuples of $d\times d$ matrix-valued functions in $\R^n$. 
Define the generalized commutator $[T_\alpha,(\mathbf A,\mathbf B)]$ of an operator $T_\alpha$ acting on vector-valued functions by
\begin{equation*}\label{Eq:generalizedcommutator}
    \left[T_\alpha,(\mathbf A,\mathbf B)\right]\vecf\coloneqq\sum_{i=1}^kA_iT_\alpha (B_i\vecf).
\end{equation*}

\subsection*{Two-weight norm inequalities}
We first prove two-matrix-weight norm inequalities for $[T,(\mathbf A,\mathbf B)]$.  To state our main result, we  introduce some notation.  
We extend the Euclidean dot product to the $k$-tuples $\mathbf A,\, \mathbf B$ by
\[
    \mathbf A(y)\cdot \mathbf B(x)\coloneqq\sum_{i=1}^kA_i(y)B_i(x).
\]
Given Young functions $\Phi$ and $\Psi$, we define the iterated norm $\avgL^\Psi_y(\avgL^\Phi_x)(Q)$ by
\[  \|f(x,y)\|_{\avgL^\Psi_y(\avgL^\Phi_x)(Q)} = \big\| \|f(x,y)\|_{\avgL^\Phi_x(Q)} \big\|_{\avgL^\Psi_y(Q)};  \]
the subscripts $x$ and $y$ indicate which variable the norms are computed with.

\begin{theorem}\label{thm:orliczbumps}
    Given $0\leq \alpha<n$ and $1<p<\frac{n}{\alpha}$, define $q$ by $\frac{1}{p}-\frac{1}{q}=\frac{\alpha}{n}$, and let $1\leq r<p$. Suppose that $\Phi$ and $\Psi$ are Young functions such that 
    \[\bar \Phi_r\in \mathcal B_{p,q},\; \bar \Psi\in \mathcal B_{q'} 
    \qquad\text{or} \qquad \bar \Phi_r\in\mathcal B_{p}, \; \bar \Psi\in\mathcal B_{q',p'},\]
    where $\bar{\Phi}_r(t)\coloneqq \bar\Phi(t^r)$.
    Suppose further that $V$ and $W$ are matrix weights such that
    \[
        \kappa\coloneqq\sup_Q\||V(y)\left[\mathbf A(y)\cdot \mathbf B(x)\right]W^{-1}(x)|_\op\|_{\avgL^\Psi_y(\avgL^{\Phi_r}_x)(Q)}<\infty,
    \]
      If $T_\alpha$ is an operator that satisfies~\eqref{Eq:CBDcommutatorsection}, then the generalized commutator $\left[T_\alpha,(\mathbf A,\mathbf B)\right]$ is bounded from $L^p(W)\to L^q(V)$ and satisfies
      \[ \|\left[T_\alpha,(\mathbf A,\mathbf B)\right]\|_{L^p(V)} 
      \leq CB(\Phi,\Psi) \kappa \|f\|_{L^p(W)},
\]
where $B(\Phi,\Psi)$ equals either $B_{p,q}(\bar{\Phi}_r) B_{q'}(\bar{\Psi})$ or $B_{p}(\bar{\Phi}_r) B_{q',p'}(\bar{\Psi})$.
\end{theorem}

\begin{proof}
By duality and a change of variables, replacing $\vecf$ with $W^{-1} \vecf$, it will suffice to show that for any $\vecf,\, \vecg \in L^\infty_c(\R^n,\R^d)$, 
\[ \big|\left\langle V\left[T_\alpha,(\mathbf A,\mathbf B)\right]W^{-1}\vecf,\vecg\right\rangle \big|
\leq CB(\Phi,\Psi) \kappa \|\vecf\|_{L^p(\R^n,\R^d)} \|\vecg \|_{L^{q'}(\R^n,\R^d)}.  
\]
By the definition of the generalized commutator, we have that
\begin{multline*}
    \left\langle V\left[T_\alpha,(\mathbf A,\mathbf B)\right]W^{-1}\vecf,\vecg\right\rangle
    =\sum_{i=1}^k\langle VA_iT_\alpha(B_iW^{-1}\vecf),\vecg\rangle\\
    =\sum_{i=1}^k\langle T_\alpha(B_iW^{-1}\vecf),A_i^* V\vecg\rangle
    =\sum_{i=1}^k\sum_{j=1}^d\langle T_\alpha([B_iW^{-1}\vecf]_j),[A_i^* V\vecg\,]_j\rangle
    =\langle T_\alpha\mathbf F,\mathbf G\rangle,
\end{multline*}
where $\mathbf F$ is the vector in $\R^{dk}$ whose components are the $dk$ components of the vectors
\[B_1W^{-1}\vecf,\dots,B_kW^{-1}\vecf, \] 
and $\mathbf G$ is defined in an analogous way. If we apply~\eqref{Eq:CBDcommutatorsection} to $T_\alpha \mathbf F$, we get 
\begin{equation}\label{Eq:CBDforgeneralcommutator}
    |\left\langle V\left[T_\alpha,(\mathbf A,\mathbf B)\right]W^{-1}\vecf,\vecg\right\rangle|
    \leq C\sum_{i=1}^{3^n}\sum_{Q\in\mathcal{S}_i}|Q|^{\frac{\alpha}{n}+1}\llangle \mathbf F\rrangle_{\avgL^r(Q)}\cdot
    \avgint_Q \mathbf G(y)\,dy.
\end{equation}
To estimate the inner product, by the definition of $\llangle \mathbf F\rrangle_{\avgL^r(Q)}$ it will suffice to give a uniform estimate of
\begin{equation*}  
    \avgint_Q\avgint_Q K(x) \mathbf F(x)\cdot \mathbf G(y)\, dx\, dy,
\end{equation*}
where $K \in \overline{\mathbf B}(r',Q)$. If we write out the dot product in terms of the constituent vectors in $\R^d$, we get
\begin{align*}
    \mathbf F(x)\cdot \mathbf G(y)
    & =\sum_{i=1}^kB_i(x)W^{-1}(x)\vecf(x)\cdot A_i(y)^* V(y)\vecg(y)\\
    & =\sum_{i=1}^kV(y)A_i(y)B_i(x)W^{-1}(x)\vecf(x)\cdot \vecg(y) \\
    & =V(y)\left[\mathbf A(y)\cdot \mathbf B(x)\right]W^{-1}(x)\vecf(x)\cdot \vecg(y).
\end{align*}
Therefore, by H\"older's inequality and the fact that $\|K\|_{\avgL^{r'}(Q)}\leq 1$, we have
\begin{align*}
  & \bigg|\avgint_Q\avgint_Q K(x) \mathbf F(x)\cdot \mathbf G(y)\, dx\, dy\bigg|   \\
  & \qquad \qquad \leq\avgint_Q\avgint_Q |K(x)| |V(y)\left[\mathbf A(y)\cdot \mathbf B(x)\right]W^{-1}(x)|_\op|\vecf(x)| \,dx |\vecg(y)|\,dy \\
  & \qquad \qquad \leq\avgint_Q\left(\avgint_Q |V(y)\left[\mathbf A(y)\cdot \mathbf B(x)\right]W^{-1}(x)|_\op^r|\vecf(x)|^r \,dx
  \right)^\frac{1}{r}|\vecg(y)|\,dy.
  \intertext{We now apply the generalized H\"older's inequality for Orlicz spaces \eqref{eq:orliczholder}, first with $\Phi$ and then with $\Psi$, and then use the rescaling property~\eqref{eqn:rescaling} to get }
  & \qquad \qquad \leq 2\avgint_Q \big\| |V(y)\left[\mathbf A(y)\cdot \mathbf B(x)\right]W^{-1}(x)|_\op^r \|_{\avgL_x^\Phi(Q)}^{\frac{1}{r}}
  \| |\vecf|^r \|_{\avgL^{\bar \Phi}(Q)}^{\frac{1}{r}}|\vecg(y)|\,dy \\
  & \qquad \qquad = 2\avgint_Q \big\| |V(y)\left[\mathbf A(y)\cdot \mathbf B(x)\right]W^{-1}(x)|_\op \|_{\avgL_x^{\Phi_r}(Q)}
  \| \vecf \|_{\avgL^{\bar \Phi_r}(Q)}|\vecg(y)|\,dy \\
  & \qquad \qquad  \leq 4\big\| |V(y)\left[\mathbf A(y)\cdot \mathbf B(x)\right]W^{-1}(x)|_\op \|_{\avgL^\Psi_y(\avgL_x^{\Phi_r})(Q)}
  \| \vecf \|_{\avgL^{\bar \Phi_r}(Q)} \|\vecg\|_{\avgL^{\bar \Psi}(Q)}.
\end{align*}
Since this bound is independent of $K$, we have shown that 
\begin{equation} \label{eqn:dot-prod-est}
 \llangle \mathbf F\rrangle_{\avgL^r(Q)}\cdot
    \avgint_Q \mathbf G(y)\,dy \leq 
    4\kappa \| \vecf \|_{\avgL^{\bar \Phi_r}(Q)} \|\vecg\|_{\avgL^{\bar \Psi}(Q)}.
    \end{equation}

To complete the proof, first suppose that $\bar \Phi_r\in \mathcal B_{p,q},\; \bar \Psi\in \mathcal B_{q'}$.  
   If we combine~\eqref{eqn:dot-prod-est} with~\eqref{Eq:CBDforgeneralcommutator}, use the fact that the sets $\Ss_i$ are $\eta$-sparse, and apply H\"older's inequality, we get that 
    \begin{align*}
        |\left\langle V\left[T_\alpha,(\mathbf A,\mathbf B)\right]W^{-1}\vecf,\vecg\right\rangle|
        &\lesssim \eta^{-1} \kappa \sum_{i=1}^{3^n} \sum_{Q\in\mathcal{S_i}}
        |Q|^{\frac{\alpha}{n}}\|\vecf\|_{\avgL^{\bar \Phi_r}(Q)}
        \|\vecg\|_{\avgL^ {\bar \Psi}(Q)} |E_Q| \\
        &\lesssim 3^n \kappa \int_{\R^n} M_{\alpha,\bar \Phi_r}\vecf(x)M_{\bar \Psi}\vecg(x)\,dx \\
        &\lesssim\kappa\|M_{\alpha,\bar \Phi_r}\vecf\|_{L^q(\R^n)}\|M_{\bar \Psi}\vecg\|_{L^{q'}(\R^n)}\\&\lesssim\kappa\, B_{p,q}(\bar \Phi_r)B_{q'}(\bar\Psi)\|\vecf\|_{L^p(\R^n,\R^d)}\|\vecg\|_{L^{q'}(\R^n,\R^d)}.
    \end{align*}
    In the last inequality we used the fractional maximal Orlicz bound \eqref{Eq:FracOrliczMaximalEst}.

    On the other hand, if $\bar \Phi_r\in\mathcal B_{p},\;\bar \Psi\in\mathcal B_{q',p'}$, then we can modify the above argument to show that
    \begin{align*}
        |\left\langle V\left[T_\alpha,(\mathbf A,\mathbf B)\right]W^{-1}\vecf,\vecg\right\rangle|
        &\lesssim \kappa\sum_{i=1}^{3^n} \sum_{Q\in\mathcal{S}_i}|Q|^{\frac{\alpha}{n}}
        \|\vecf\|_{\avgL^{\bar \Phi_r}(Q)}\|\vecg\|_{\avgL^{\bar \Psi}(Q)} |E_Q|\\
        &\leq\kappa\|M_{\bar \Phi_r}\vecf\|_{L^p(\R^n)}\|M_{\alpha,\bar \Psi}\vecg\|_{L^{p'}(\R^n)}\\ 
        &\lesssim\kappa\, B_{p}(\bar \Phi_r)B_{q',p'}(\bar\Psi)\|\vecf\|_{L^p(\R^n,\R^d)}\|\vecg\|_{L^{q'}(\R^n,\R^d)}.
    \end{align*}
    This completes the proof.
\end{proof}

\begin{remark}\label{remark:iteratednorms}
    It is clear from the proof of Theorem \ref{thm:orliczbumps} that in the case $r=1$ one could also take the iterated norm in a different order. In particular, with $\Phi,\,\Psi$ as in Theorem \ref{thm:orliczbumps}, we have that
    \[
       \sup_Q\|V(y)\left[\mathbf A(y)\cdot \mathbf B(x)\right]W^{-1}(x)\|_{\avgL^\Phi_x(\avgL^\Psi_y)(Q))}<\infty
    \]
    also implies $L^p(W)\to L^q(V)$ boundedness for $[I_\alpha,(\mathbf{A,B})]$.
\end{remark}

As a corollary of Theorem \ref{thm:orliczbumps} we have the following power bump result.

\begin{cor}\label{cor:powerbumps}
  Given $0 \leq \alpha<n$ and $1<p<\frac{n}{\alpha}$, define $q$ by $\frac{1}{p}-\frac{1}{q}=\frac{\alpha}{n}$, and let $1\leq r<p$. Let $u,\,v>1$  and suppose that $V,\,W$ are matrix weights such that
    \[
        \sup_Q\left(\avgint_Q\left[\avgint_Q |V(y)\left[\mathbf A(y)\cdot \mathbf B(x)\right]W^{-1}(x)|_{\op}^{u(\frac{p}{r})'}\, dx\right]^\frac{vq}{u(\frac{p}{r})'} \,dy\right)^\frac{1}{vq}<\infty.
    \]
    If $T_\alpha$ is an operator that satisfies~\eqref{Eq:CBDcommutatorsection}, then the generalized commutator $\left[T_\alpha,(\mathbf A,\mathbf B)\right]$ is bounded from $L^p(W)\to L^q(V)$.
\end{cor}

\begin{proof}
    The Young functions $t \mapsto t^{u(\frac{p}{r})'}$ and  $t\mapsto t^{vq}$ satisfy the hypotheses on $\Phi$ and $\Psi$ in Theorem~\ref{thm:orliczbumps}, which immediately gives us the desired result.
\end{proof}


We now show how particular choices of $A_i,\,B_i$ yield specific examples of commutators.

\begin{example}[Fixed symbol $B$] \label{Ex:Commutatorswithasinglesymbol}
    The choices $r=1$, $T_\alpha=I_\alpha$, $\mathbf{A}=(B,-\operatorname{Id})$ and $\mathbf{B}=(\operatorname{Id},B)$ gives us the first order commutator \[[I_\alpha,B]\vecf\coloneqq BI_\alpha\vecf-I_\alpha(B\vecf).\] 
\end{example}    

\begin{remark}
    In this case Theorem \ref{thm:orliczbumps} recovers \cite[Proposition 1.5]{CI2022twomatrixweightedfracint}.
\end{remark}

\begin{example}
    For each non-negative integer $k$, if we let $A_j=\binom{k}{j}B^{k-j}$ and $B_j=(-B)^j$, $j=1,\dots,k$, then we get the $k$-th order iterated commutator 
    \[
    C_B^k(I_\alpha)\coloneqq[B,C_B^{k-1}(I_\alpha)],\qquad C^0_B(I_\alpha)\coloneqq I_\alpha. \] 
\end{example}

\begin{example}[Mixed symbols]\label{Ex:Commutatorswithmultiplesymbols}
Define $k=2^m$, $\mathbf A=\{A_\sigma\}_{\sigma\in\{0,1\}^m}$, $\mathbf B=\{B_\sigma\}_{\sigma\in\{0,1\}^m}$, where
    \[
        A_\sigma(y)\coloneqq \prod_{j\colon \sigma_j=0}B^{(j)}(y)
    \]
    and
    \[
        B_\sigma(x)\coloneqq (-1)^{|\sigma|}\prod_{j\colon \sigma_j=1}B^{(j)}(x).
    \]
    With these values we get  the mixed iterated commutators 
    \[C_{B^{(1)},\dots,B^{(m)}}(T_\alpha)\coloneqq[B^{(m)},C_{B^{(1)},\dots,B^{(m-1)}}(T_\alpha)],
    \qquad C_{B^{(1)}}(T_\alpha)\coloneqq[B^{(1)},T_\alpha]. \] 
    \end{example}
    
    \begin{remark} \label{remark:irrsr0}
  In this case, Theorem \ref{thm:orliczbumps} is an analogue of \cite[Theorem 12]{IRS2026convexbodydominationcommutator}, which was proved for Calder\'on-Zygmund operators (i.e., the diagonal case when $\alpha=0$).
    \end{remark}
    
    In \cite{HLO2020iteratedcommutators,hyt2024remarksonCBD} they considered iterated commutators with $m=2$ and  scalar symbols for Calder\'on-Zygmund operators $T$. In \cite{HLO2020iteratedcommutators} they showed that given  Young functions $\Phi,\Psi, \Xi$ such that their associate functions are in  $\mathcal B_2$,
    \[
        \sup_Q\|b^{(1)}-\langle b^{(1)}\rangle_Q\|_{\Phi(Q)}\|b^{(2)}-\langle b^{(2)}\rangle_Q\|_{\Psi(Q)}<\infty
    \]
    and 
    \[
        \sup_Q\|(b^{(1)}-\langle b^{(1)}\rangle_Q)(b^{(2)}-\langle b^{(2)}\rangle_Q)\|_{\Xi(Q)}<\infty,
    \]
    then 
    \[
        C_{b^{(1)},b^{(2)}}(T):L^2(\R^n)\to L^2(\R^n).
    \]
    As we noted above, we can apply Theorem~\ref{thm:orliczbumps} with $\alpha=0$, that is, to Calder\'on-Zygmund operators.  In particular if we assume that $\bar\Phi,\bar\Psi\in\mathcal B_{2}$, then the condition to get boundedness in $L^2(\R^n)$ is
    \begin{equation} \label{eqn:alt-cond}
        \kappa=\sup_Q\|(b^{(1)}(y)-b^{(1)}(x))(b^{(2)}(y)-b^{(2)}(x))\|_{\avgL^\Phi_x(\avgL^\Psi_y)(Q)}<\infty.
    \end{equation}
    In \cite[Example 7.8]{hyt2024remarksonCBD} they showed that, if $\Phi$ and $\Psi$ are power bumps (as in Corollary~\ref{cor:powerbumps}), then  condition~\eqref{eqn:alt-cond} is at least as sharp as the one in \cite{HLO2020iteratedcommutators}. 

\subsection*{One-weight norm inequalities}
We now consider one-matrix-weight norm inequalities.  For simplicity, we will restrict our attention to 
the  $k$-th order iterated commutator with  scalar symbol $b$.

\begin{theorem}\label{thm:MatWeightCommutatorBound}
 Given $0 \leq \alpha<n$ and $1<p<\frac{n}{\alpha}$, define $q$ by $\frac{1}{p}-\frac{1}{q}=\frac{\alpha}{n}$, and let $1\leq r<p$. Let $k$ be a nonnegative integer and suppose that $W^r\in \A_{\frac{p}{r},\frac{q}{r}}$ and $b\in \operatorname{BMO}$.  If $T_\alpha$ is an operator that satisfies~\eqref{Eq:CBDcommutatorsection}, then we have 
\begin{equation} \label{eqn:MWCB1}
        \|WC^k_b(T_\alpha)\vecf\|_{L^q(\R^n,\R^d)}\leq C\|b\|_{BMO}^k [W^r]_{\A_{\frac{p}{r},\frac{q}{r}}}^{\beta}\|W\vecf\|_{L^p(\R^n,\R^d)},
        \end{equation}
    where 
    \[
        \beta\coloneqq k\max\left\{\left(\frac{p}{r}\right)',\frac{q}{r}\right\}+\min\left\{\left(\frac{p}{r}\right)'\frac{1}{q}+\frac{q}{rq'},\,\left(\frac{p}{r}\right)'\frac{1}{p}+\frac{q}{rp'}\right\}+\frac{1}{r}
    \]
    and the constant $C$ depends only on fixed parameters $k,p,q,r,n,d$ and $\eta$.
\end{theorem}

\begin{remark}
    When $r=1$, Theorem~\ref{thm:MatWeightCommutatorBound} reduces to Theorem~\ref{thm:MatWeightCommutatorBoundIntro}. To see that the exponent $\beta$ is the same in both statements, note that since $\frac{1}{p}-\frac{1}{q}=\frac{\alpha}{n}$, $1-\frac{\alpha}{n}=\frac{1}{p'}+\frac{1}{q}$.  Furthermore, if we let  $k=0$, then we get a norm inequality for the fractional integral with exponent $\min\{q+p'(1-\frac{\alpha}{n}),p'+(1-\frac{\alpha}{n})q\}-1$.  In other words, we get weighted estimates for fractional integrals in Theorem~\ref{thm:I-alpha-bound-intro}.
\end{remark}

\begin{remark} \label{remark:IRRSR}
Theorem~\ref{thm:MatWeightCommutatorBound} is a quantitative, off-diagonal, one-weight counterpart of \cite[Theorem 6]{IRS2026convexbodydominationcommutator}, which gives a qualitative, two-weight estimate in the diagonal case (i.e., when $\alpha=0$) when both weights are in $\A_p$ and $r=1$.  Note that in the one-weight case, their theorem gives the same power on the $BMO$ norm of $b$.  (See the discussion following the statement of Theorem~6.)
\end{remark}

\begin{proof}
By Theorem \ref{thm:orliczbumps} we have
\[
    \|WC^k_b(T_\alpha)\vecf\|_{L^q(\R^n,\R^d)}\lesssim\kappa B(\Phi,\Psi)\|W\vecf\|_{L^p(\R^n,\R^d)},
\]
where
\begin{equation}\label{Eq:BloomOrlicz}
    \kappa=\sup_Q\|W(y)W^{-1}(x)(b(y)-b(x))^k\|_{\avgL^{\Psi}_y(\avgL^{\Phi_r}_x)(Q)}.
\end{equation}
We will first estimate $\kappa$ and then estimate $B(\Phi,\Psi)$.  
Let $\Phi$ and $\Psi$  be the power bumps defined in Corollary~\ref{cor:powerbumps}, i.e, $\Psi(t)=t^{vq}$ and $\Phi(t)=t^{u(\frac{p}{r})'}$ where $u,v>1$ are defined by 
\[
    v\coloneqq1+\tilde\varepsilon,\qquad\tilde\varepsilon
    \coloneqq\frac{1}{2\tau(n)[W^r]_{\A_{\frac{p}{r},\frac{q}{r}}}^\frac{q}{r}-2},
\]
and
\[
    u\coloneqq1+\tilde\delta,\quad \tilde\delta
    \coloneqq\frac{1}{2\tau(n)[W^{-r}]_{\A_{(\frac{q}{r})',(\frac{p}{r})'}}^{(\frac{p}{r})'}-2};
\]
here $\tau(n)$ is the constant from Lemma~\ref{lemma:matrix-RHI}.  (The reason for these choices will be made clear below.)

By the triangle inequality $\kappa\lesssim \kappa_1+\kappa_2$,
where
\[
    \kappa_1\coloneqq\sup_Q\||W(y)W^{-1}(x)|_\op (b(y)-\langle b\rangle_Q)^k\|_{\avgL^{\Psi}_y(\avgL^{\Phi_r}_x)(Q)}
\]
and 
\[
    \kappa_2\coloneqq\sup_Q\||W(y)W^{-1}(x)|_\op (b(x)-\langle b\rangle_Q)^k\|_{\avgL^{\Psi}_y(\avgL^{\Phi_r}_x)(Q)}.
\]

We estimate $\kappa_1$ and $\kappa_2$ separately.    
Let  $\overline{\W}^{\Phi,r}_Q, \,\W^{\frac{q}{r},r}_Q$ be the reducing matrices that satisfy, for all $\vecv\in \R^d$, 
\[
    |\overline{\W}^{\Phi,r}_Q\vecv| \approx 
    \|W^{-r} \vecv\|_{\avgL^\Phi(Q)}=\|W^{-r}\vecv\|_{\avgL^{u(\frac{p}{r})'}},
\]
and
\[
    |\W^{\frac{q}{r},r}_Q \vecv| \approx \|W^r\vecv\|_{\avgL^{\frac{q}{r}}(Q)}.
\]
By Lemmas~\ref{SelfAdjointCommutes} and~\ref{opNorm:equiv}, we have the following well-known relation:
\begin{multline*}  \||(\overline{\W}^{\Phi,r}_Q)^{-1}W^{-r}|_\op\|_{\avgL^\Phi(Q)}
= \||W^{-r}(\overline{\W}^{\Phi,r}_Q)^{-1}|_\op\|_{\avgL^\Phi(Q)} \\
\approx \sum_{i=1}^d \|W^{-r}(\overline{\W}^{\Phi,r}_Q)^{-1}\vece_i\|_{\avgL^\Phi(Q)}
\approx \sum_{i=1}^d |\overline{\W}^{\Phi,r}_Q(\overline{\W}^{\Phi,r}_Q)^{-1}\vece_i|
= d
\end{multline*}
By the Cordes inequality, Lemma \ref{lemma:cordes}, we have that
\[
    |W(y)W^{-1}(x)|_{\op}
    \leq|W^r(y)W^{-r}(x)|_{\op}^\frac{1}{r}
    \leq |W^r(y)\overline{\W}^{\Phi,r}_Q|_{\op}^\frac{1}{r}
    |(\overline{\W}^{\Phi,r}_Q)^{-1}W^{-r}(x)|_{\op}^\frac{1}{r}.
\]

We can now estimate as follows:  by the definition of the iterated norm and by~\eqref{eqn:rescaling},
\begin{align*}
    \kappa_1
    &\leq\sup_Q\||W^r\overline{\W}^{\Phi,r}_Q|_{\op}^\frac{1}{r}(b-\langle b\rangle_Q)^k\|_{\avgL^\Psi(Q)}\|
    |(\overline{\W}^{\Phi,r}_Q)^{-1}W^{-r}|_{\op}^\frac{1}{r}\|_{\avgL^{\Phi_r}(Q)}\\
    &=\sup_Q\||W^r\overline{\W}^{\Phi,r}_Q|_{\op}^\frac{1}{r}(b-\langle b\rangle_Q)^k\|_{\avgL^\Psi(Q)}
    \||(\overline{\W}^{\Phi,r}_Q)^{-1}W^{-r}|_\op\|_{\avgL^\Phi(Q)}^\frac{1}{r}\\
    &\approx \sup_Q\||W^r\overline{\W}^{\Phi,r}_Q|_{\op}^\frac{1}{r}(b-\langle b\rangle_Q)^k\|_{\avgL^\Psi(Q)} \\
    & =\sup_Q\||W^r\overline{\W}^{\Phi,r}_Q|_{\op}^\frac{1}{r}(b-\langle b\rangle_Q)^k\|_{\avgL^{vq}(Q)}.
\end{align*}

As a consequence of the John-Nirenberg inequality (see~\cite[Lemma~7.2]{Lau2025multiscale}) we have the quantitative estimate $\|b-\langle b \rangle_Q \|_{\avgL^p(Q)} \lesssim p \|b\|_{BMO}$, $1\leq p<\infty$.  
Therefore, by  Hölder's inequality, we have that
\begin{align*}
   \||W^r\overline{\W}^{\Phi,r}_Q|_{\op}^\frac{1}{r}(b-\langle b\rangle_Q)^k\|_{\avgL^{vq}(Q)}
   &\leq\||W^r\overline{\W}^{\Phi,r}_Q|_{\op}^\frac{1}{r}\|_{\avgL^\frac{vq}{1-\varepsilon}(Q)}\|b-\langle b\rangle_Q\|_{\avgL^\frac{vqk}{\varepsilon}(Q)}^k \\
   &\lesssim \||W^r\overline{\W}^{\Phi,r}_Q|_\op\|_{\avgL^{\frac{v}{1-\varepsilon}\frac{q}{r}}(Q)}^\frac{1}{r}
   \varepsilon^{-k}\|b\|_{BMO}^k.
   \end{align*}

Let $\varepsilon\coloneqq\frac{\tilde\varepsilon}{1+2\tilde\varepsilon}$. Then we have 
\[ \frac{v}{1-\varepsilon}=1+\frac{1}{\tau(n)[W^r]_{\A_{\frac{p}{r},\frac{q}{r}}}^\frac{q}{r}-1}\quad
\text{and}\quad \frac{1}{\varepsilon}
\lesssim \frac{1}{\tilde\varepsilon}\lesssim [W^r]_{\A_{\frac{p}{r},\frac{q}{r}}}^\frac{q}{r}.\] 
Note that $\frac{v}{1-\varepsilon}$ is the exponent in the sharp reverse H\"older inequality for $|W^r\vecv|^{\frac{q}{r}}$.  Therefore, by Lemmas~\ref{opNorm:equiv} and~\ref{lemma:matrix-RHI},
   \begin{align*}      
 \||W^r\overline{\W}^{\Phi,r}_Q|_\op\|_{\avgL^{\frac{v}{1-\varepsilon}\frac{q}{r}}(Q)}^\frac{1}{r}
   \varepsilon^{-k}\|b\|_{BMO}^k
    & \lesssim \sum_{i=1}^d \||W^r\overline{\W}^{\Phi,r}_Q\vece_i|\|_{\avgL^{\frac{v}{1-\varepsilon}\frac{q}{r}}(Q)}^\frac{1}{r}
   \varepsilon^{-k}\|b\|_{BMO}^k \\   
   & \lesssim \sum_{i=1}^d \||W^r\overline{\W}^{\Phi,r}_Q\vece_i|\|_{\avgL^{\frac{q}{r}}(Q)}^\frac{1}{r}
   \varepsilon^{-k}\|b\|_{BMO}^k \\   
   & \lesssim\||W^r\overline{\W}^{\Phi,r}_Q|_\op\|_{\avgL^{\frac{q}{r}}(Q)}^\frac{1}{r}
   \varepsilon^{-k}\|b\|_{BMO}^k \\
&\lesssim |\W^{\frac{q}{r},r}_Q\,\overline{\W}^{\Phi,r}_Q|_{\op}^\frac{1}{r}[W^r]_{\A_{\frac{p}{r},\frac{q}{r}}}^{k\frac{q}{r}}\|b\|_{BMO}^k.
\end{align*}
To estimate the first term in the last line, we apply Lemma~\ref{SelfAdjointCommutes} and the definition of the reducing operators,  and repeat the above argument using the reverse H\"older inequality to get 
\begin{multline*}
    |\W^{\frac{q}{r},r}_Q\,\overline{\W}^{\Phi,r}_Q|_{\op}
    = |\overline{\W}^{\Phi,r}_Q\, \W^{\frac{q}{r},r}_Q|_{\op} 
    \approx\||W^{-r}\W^{\frac{q}{r},r}_Q|_\op\|_{\avgL^{u(\frac{p}{r})'}(Q)} \\
    \lesssim\||W^{-r}\W^q_Q|_\op\|_{\avgL^{(\frac{p}{r})'}(Q)}
    \approx |\W^{\frac{q}{r},r}_Q\,\overline{\W}_Q^{(\frac{p}{r})',r}|_{\op} 
    \approx [W^r]_{\A_{\frac{p}{r},\frac{q}{r}}};
\end{multline*}
the last inequality follows by Proposition \ref{prop:frac-reducing}.  Therefore, if we combine the above estimates, we get that
\[
    \kappa_1\lesssim [W^r]_{\A_{\frac{p}{r},\frac{q}{r}}}^{k\frac{q}{r}+\frac{1}{r}}\|b\|_{BMO}^k.
\]

To estimate $\kappa_2$, let $\delta\coloneqq\frac{\tilde\delta}{1+2\tilde\delta}$ and argue as we did above to get 
\begin{align*}
    \kappa_2
    &\lesssim\sup_Q\||W^{-r}\W^{\frac{q}{r},r}_Q|_{\op}^\frac{1}{r}(b-\langle b\rangle_Q)^k\|_{\avgL^{ru(\frac{p}{r})'}(Q)}\\ 
    &\leq\||W^{-r}\W^{\frac{q}{r},r}_Q|_\op\|_{\avgL^{\frac{u}{1-\delta}(\frac{p}{r})'}(Q)}^\frac{1}{r}\|b-\langle b\rangle_Q\|_{\avgL^\frac{ru(\frac{p}{r})'k}{\delta}(Q)}^k \\ 
    &\lesssim |\overline{\W}_Q^{(\frac{p}{r})',r}\,\W^{\frac{q}{r},r}_Q|_{\op}^\frac{1}{r}[W^{-r}]_{\A_{(\frac{q}{r})',(\frac{p}{r})'}}^{k(\frac{p}{r})'}\|b\|_{BMO}^k \\ 
    &\lesssim [W^r]_{\A_{\frac{p}{r},\frac{q}{r}}}^{k(\frac{p}{r})'+\frac{1}{r}}\|b\|_{BMO}^k;
\end{align*}
in the last estimate we have used Remark~\ref{remark:Apq-dual}.
If we compare the estimates for $\kappa_1$ and $\kappa_2$ and take the larger exponent, we get that
\[ \kappa \lesssim [W^r]_{\A_{\frac{p}{r},\frac{q}{r}}}^{k\max\{(\frac{p}{r})',\frac{q}{r}\}+\frac{1}{r}}\|b\|_{BMO}^k. \]

\medskip

We now  estimate $B(\Phi,\Psi)$.  From  our definition of  $\Phi$ and $\Psi$ as power bumps, we have that 
\[  \bar{\Phi}_r \in \cB_p \subset \cB_{p,q}  \quad \text{ and } \quad \bar{\Psi} \in \cB_{q'} \subset \cB_{q',p'}.  \]
Therefore, we can use either hypothesis on the Young functions in Theorem~\ref{thm:orliczbumps}, and so we have that
\[  B(\Phi,\Psi) = \min\{ B_{p,q}(\bar\Phi_r)B_{q'}(\bar\Psi), B_{p}(\bar\Phi_r)B_{q',p'}(\bar\Psi)\}.  \]
We estimate these quantities using~\eqref{eqn:Bpq-cond}.  Since $\bar \Phi_r(t) \approx t^{r(u(p/r)')'} = t^{\frac{upr}{up+r-p}}$, and  since $r- p$ we have $\frac{upr}{up+r-p}<q$.  Therefore,
\begin{multline*}
    B_{p,q}(\bar\Phi_r) 
    = \bigg(\int_1^\infty \frac{{\bar \Phi}_r(t)^{\frac{q}{p}}}{t^{q+1}}\,dt\bigg)^{\frac{1}{q}} 
    \approx \bigg(\int_1^\infty t^{\frac{uqr}{up+r-p}-q-1}\,dt\bigg)^{\frac{1}{q}} 
    = \bigg(\frac{-1}{\frac{uqr}{up+r-p}-q}\bigg)^{\frac{1}{q}} \\
    = \bigg(\frac{1}{q} \frac{u}{(u-1)}\frac{p}{(p-r)} \bigg)^{\frac{1}{q}}
    \lesssim (u')^{\frac{1}{q}}
    \lesssim [W^{-r}]_{\A_{(\frac{q}{r})',(\frac{p}{r})'}}^{(\frac{p}{r})'\frac{1}{q}}
    \approx [W^{r}]_{\A_{\frac{p}{r},\frac{q}{r}}}^{(\frac{p}{r})'\frac{1}{q}}.
\end{multline*}
The last equivalence follows from Remark~\ref{remark:Apq-dual}.  Similar computations show that
\begin{gather*}
    B_{q'}(\bar\Psi)\lesssim (v')^\frac{1}{q'}\lesssim [W^r]_{\A_{\frac{p}{r},\frac{q}{r}}}^\frac{q}{rq'}, \\
    B_{p}(\bar\Phi_r)\lesssim (u')^\frac{1}{p}\lesssim [W^r]_{\A_{\frac{p}{r},\frac{q}{r}}}^{(\frac{p}{r})'\frac{1}{p}}, \\
    B_{q',p'}(\bar\Psi)\lesssim (v')^\frac{1}{p'}\lesssim [W^r]_{\A_{\frac{p}{r},\frac{q}{r}}}^\frac{q}{rp'}.    
\end{gather*}
If we combine these estimates, we see that 
\[ B(\Phi,\Psi) \lesssim [W^{r}]_{\A_{\frac{p}{r},\frac{q}{r}}}^{\min\{(\frac{p}{r})'\frac{1}{q} + \frac{q}{rq'}, 
(\frac{p}{r})'\frac{1}{p} + \frac{q}{rp'}\} }.  
\]
This, together with our estimate for $\kappa$ above, gives the desired constant in~\eqref{eqn:MWCB1}.  This completes the proof.
\end{proof}

\section{Matrix-Weighted Sobolev Inequalities}
\label{section:sobolev}

In this section we prove matrix-weighted Gagliardo-Nirenberg-Sobolev (GNS) inequalities for vector-valued functions.  The classical GNS inequality for scalar functions is
$$\|f\|_{L^{p^*}(\R^n)}\leq C\|\nabla f\|_{L^p(\R^n)}, \qquad f\in \dot W^{1,p}(\R^n),$$
where $1\leq p<n$ and $p^*=\frac{np}{n-p}$.  (Here, $\dot W^{1,p}(\R^n)$ is the homogeneous Sobolev space of functions that satisfy $\|\grad f\|_{L^p(\R^n,\R^d)}<\infty$.) 
Muckenhoupt and Wheeden~\cite{muckenhoupt-wheeden74} proved a weighted GNS using the weighted inequality for  the fractional integral:  given  $0<\alpha<n$ and $1<p<\frac{n}{\alpha}$, if $\frac{1}{p}-\frac{1}{q}=\frac{\alpha}{n}$ and $w\in \A_{p,q}$, then 
$$\|wI_\alpha f\|_{L^q(\R^n)}\leq C\| wf\|_{L^p(\R^n)}.$$
Given this, it follows immediately from  the subrepresentation formula
\begin{equation}\label{subrep}
|f(x)|\leq CI_1(|\nabla f|)(x),\qquad f\in C_c^\infty(\R^n),
\end{equation}
that if $w\in \A_{p,p^*}$, then we have the weighted GNS inequality
$$\|f\|_{L^{p^*}(w)}\leq C\|I_1(|\nabla f|)\|_{L^{p^*}(w)}\leq C\|\nabla f\|_{L^p(w)}. $$

In the vector-valued setting this approach using the subrepresentation formula does not work since it is not clear how to bound $|W(x)f(x)|$ in terms of the fractional integral $I_1$.  In \cite{MR4030471} they used the  representation formula
\begin{equation}\label{represent} 
f(x)=\frac{1}{\omega_{n-1}}\int_{\R^n}\frac{x-y}{|x-y|^n}\cdot \nabla f(y)\,dy,\qquad f\in C_c^\infty(\R^n).
\end{equation}
(This formula is sometimes referred to as the Bogovskii integral; also see~\cite[Section~V.2]{MR0290095}.)
Define the vector-valued fractional integral $\mathbf I_1$ by
$$\mathbf I_1f(x)=\frac{1}{\omega_{n-1}}\int_{\R^n}\frac{x-y}{|x-y|^n}f(y)\,dy,$$
and denote the component operators by $\mathbf I_1=(I_1^1,\ldots,I_1^n)$.  With this, we see that~\eqref{represent} can be rewritten as
$$f=\mathbf I_1\cdot \nabla f=\sum_{j=1}^n I_1^j(\partial_j f).$$
Since
$$\frac{|x_j-y_j|}{|x-y|^n}\leq \frac{1}{|x-y|^{n-1}},$$
the operators $I^j_1$ are fractional singular integral operators as defined above in~\eqref{Eq:genFracInt}. In particular, the $I_1^j$ satisfy  sparse bounds and are  bounded from $L^p(W)$ to $L^{p^*}(W)$ if $W\in \A_{p,p^*}$:  see Corollary~\ref{Cor:MTalpha-bd}.

\begin{proof}[Proof of Theorem~\ref{thm:GNS-intro}]
By a standard density argument, it suffices to prove this for a vector-valued function  $\vecf\in C_c^\infty(\R^n,\R^d)$.   If we apply~\eqref{represent} on each component, we get
$$ \vecf(x)=\sum_{j=1}^n {I}_1^j(\partial_j\vecf)(x).$$
Therefore, given a matrix weight $W\in \A_{p,p^*}$, we have 
$$W(x)\vecf(x)=\sum_{j=1}^nW(x){I}_1^j(\partial_j\vecf)(x),$$
and so by the boundedness of $I_1^j$ on matrix-weighted spaces,
\begin{align*}
    \|W\vecf\|_{L^{p^*}(\R^n,\R^d)}
    & \leq \sum_{j=1}^n\|W{I}_1^j(\partial_j\vecf)\|_{L^{p^*}(\R^n,\R^d)}\\
    &  \leq [W]_{\A_{p,p^*}}^{\min\{\frac{p'}{n'}+p^*,\frac{p^*}{n'}+p'\}-1}\sum_{j=1}^n\|W\partial_j\vecf\|_{L^p(\R^n,\R^d)} \\
    &  \leq [W]_{\A_{p,p^*}}^{\min\{\frac{p'}{n'}+p^*,\frac{p^*}{n'}+p'\}-1}\|WD\vecf\|_{L^p(\R^n,\R^{d\times n})},
\end{align*}
where $D\vecf= [\partial_j f_i]$ is the derivative matrix of $\vecf$.
\end{proof}

\medskip

We can immediately generalize Theorem~\ref{thm:GNS-intro}.  Given  $0<\alpha<1$, define the Riesz fractional derivative or fractional gradient (see \cite[Section~15.2]{MR3675703}) by
$$\nabla^\alpha f(x)=-\operatorname{p.v.}\int_{\R^n}\frac{f(y)}{|x-y|^{n +\alpha}}\frac{x-y}{|x-y|}\,dy=\int_{\R^n}\frac{f(x)-f(y)}{|x-y|^{n +\alpha}}\frac{x-y}{|x-y|}\,dy;$$
for the second equality, we have used the fact that 
$$\operatorname{p.v.}\int_{\mathbb R^{n}}\frac{x-y}{|x-y|^{n+\alpha+1}}\,dy=0$$
to add a term in the principal value integral.  We make this change because  the final integral converges absolutely if $f\in C_c^\infty(\R^n)$ and so we get
$$|\nabla^\alpha f(x)|\leq \int_{\R^n}\frac{|f(x)-f(y)|}{|x-y|^{n +\alpha}}\,dy.$$
The term on the right-hand side is the non-linear fractional derivative used to define the Gagliardo seminorm:
$$[f]_{W^{\alpha,1}(\R^n)}=\int_{\R^n}\int_{\R^n}\frac{|f(x)-f(y)|}{|x-y|^{n +\alpha}}\,dydx;$$
hence, we have that
$$\|\nabla^\alpha f\|_{L^1(\R^n)}\leq [f]_{W^{\alpha,1}(\R^n)}.$$

Denote the components of $\nabla^\alpha$ by 
$\nabla^\alpha =(\partial_1^\alpha,\ldots,\partial_n^\alpha )$. 
If we compute the derivatives of the kernel of the operator 
$I_{1-\alpha}$, we get
$$\partial_j |x|^{1-\alpha-n}=(1-\alpha-n)\frac{x_j}{|x|^{n+\alpha+1}}.$$
Therefore, we have that
$$\nabla^\alpha f(x)=c\nabla I_{1-\alpha}f(x).$$

With the fractional gradient we can generalize equation~\eqref{represent} to get the following representation formula.   (For a proof, see~\cite[Proposition 15.8]{MR3675703}.)

\begin{prop} 
Suppose $d\geq 2$ and $0<\alpha<1$.  If $f\in C_c^\infty(\R^n)$, then there exists a constant $c=c(\alpha,n)$ such that for every $x\in \R^n$,
\begin{equation} \label{eqn:frac-rep}
f(x)=c\int_{\R^n}\frac{x-y}{|x-y|^{n-\alpha+1}}\cdot\nabla^\alpha f(y)\,dy.
\end{equation}
\end{prop}

We now proceed as we did above.  For $j=1,\ldots,n$, define the operator  $I_\alpha^j$ by
$${I}_\alpha^jg(y)=\int_{\R^n}\frac{x_j-y_j}{|x-y|^{n-\alpha+1}}g(y)\,dy.$$
For each $j$ the kernel of $I_\alpha^j$ satisfies
$$K^j_\alpha(x,y)=\frac{x_j-y_j}{|x-y|^{n-\alpha+1}},$$
Thus, each $I_\alpha^j$ is a fractional singular integral as defined by~\eqref{Eq:genFracInt}.  Thus, for $W\in \A_{p,q}$, where $\frac{1}{p}-\frac{1}{q}=\frac{\alpha}{n}$, we have that it is bounded, $L^p(W)$ to $L^q(W)$.  Moreover, if we apply~\eqref{eqn:frac-rep} to each component, we get
$$\vecf(x)=\sum_{j=1}^n I_\alpha^j(\partial^\alpha_j\vecf)(x).$$
Finally, given a vector-valued function $\vecf=(f_1,\ldots,f_d)$,  define the total fractional derivative
$${D}^\alpha \vecf(x)=[\partial_1^\alpha \vecf \ \cdots \ \partial^\alpha_d\vecf]=\left[\begin{array}{c}\nabla^\alpha f_1 \\ \vdots \\ \nabla^\alpha f_d\end{array}\right]=\Big[\partial^\alpha_jf_i\Big].$$
Therefore, if we repeat the argument used above for the proof of Theorem~\ref{thm:GNS-intro}, we get the following result.

\begin{theorem}\label{Thm:MatrixWeightedFractionalGradientBound}
Given $0<\alpha<1$ and $1<p<\frac{n}{\alpha}$,  define $q$ by $\frac1q=\frac1p-\frac{\alpha}{n}$.  If $W\in \A_{p,q}$, then for all $\vecf$ such that $D^\alpha \vecf \in L^p(W)$, 
$$\|W\vecf\|_{L^q(\R^n,\R^d)}
\leq c[W]_{\mathcal \A_{p,q}}^{\min\{(1-\frac{\alpha}{n})p'+q,p'+(1-\frac\alpha{n})q\}-1}
\|W{D}^\alpha \vecf\|_{L^p(\R^n,\R^{d\times n})}.$$
\end{theorem}

\bibliographystyle{plain}
\bibliography{convex-sparse-frac}

\end{document}